\pdfoutput=1
\PassOptionsToPackage{colorlinks,citecolor=blue}{hyperref}
\documentclass[
	english
]{scrartcl}
\usepackage[T1]{fontenc}
\usepackage[utf8]{inputenc}

\usepackage{amsmath,amsfonts,amsthm,amssymb}
\usepackage{xcolor}

\usepackage{preamble}
\usepackage{marginnote}

\usepackage{enumitem}
\setlist{itemsep=0em} % Condense lists
\setlist[enumerate]{label=(\roman*)}

\usepackage[capitalise,nameinlink]{cleveref}

\usepackage{dsfont}

\usepackage{soul,cancel}

\usepackage{graphicx}

\usepackage{mathtools} % For \shortintertext

\usepackage{microtype} % Takes care of some overful \hbox'es
\usepackage{lmodern}   % Otherwise, microtype does not work

\newif\ifbiber
\bibertrue
\ifbiber
\usepackage[
	style=authoryear-comp,
	sorting=nyt,
	url=true, 
	doi=true,                % Print the DOI.
	eprint=true,
	uniquename=allinit,        % Disambiguate between D and G Wachsmuth
	maxbibnames=99,          % Do not use ``et al'' in the bibliography -- now: only with 99+ authors ;)
	maxcitenames=2,
	uniquelist=minyear,
	dashed=false,            % Do not use ---- to replace the authors in the bibliography (if the previous entry has the same authors)
	sortcites=false,          % Do not sort if multiple keys in one \cite{}
	backend=biber,           % Use biber as backend
	safeinputenc,            % For some accents, see http://tex.stackexchange.com/questions/170562
]{biblatex}

\DeclareCiteCommand{\cite}{%
	\ifbibmacroundef{cite:init}{}{\usebibmacro{cite:init}}\usebibmacro{prenote}%
}{%
	\usebibmacro{citeindex}%
	\printtext[bibhyperref]{\usebibmacro{cite}}%
}{%
	\ifbibmacroundef{cite:init}{\multicitedelim}{}%
}{%
	\usebibmacro{postnote}%
}%
\DeclareCiteCommand{\parencite}[\mkbibbrackets]{%
	\ifbibmacroundef{cite:init}{}{\usebibmacro{cite:init}}\usebibmacro{prenote}%
}{%
	\usebibmacro{citeindex}%
	\printtext[bibhyperref]{\usebibmacro{cite}}%
}{%
	\ifbibmacroundef{cite:init}{\multicitedelim}{}%
}{%
	\usebibmacro{postnote}%
}%

\let\cite\parencite

\DeclareNameAlias{sortname}{last-first}
\DeclareNameAlias{default}{last-first}

\else

\usepackage{doi,natbib}

\fi

\newcommand\norm[1]{\lVert#1\rVert}

\newcommand\abs[1]{\lvert#1\rvert}

\newcommand\dual[2]{\langle #1, #2\rangle}
\newcommand\bigdual[2]{\bigl\langle #1, #2\bigr\rangle}
\newcommand\Bigdual[2]{\Bigl\langle #1, #2\Bigr\rangle}
\newcommand\scalarprod[2]{( #1, #2)}

\newcommand{\Uad}{U_{\mathrm{ad}}}
\newcommand{\NNUad}{\NN_{\Uad}}

\newcommand\N{\mathbb{N}}

\newcommand\R{\mathbb{R}}
\newcommand\Q{\mathbb{Q}}

\newcommand\KK{\mathcal{K}}
\newcommand\LL{\mathcal{L}}
\newcommand\MM{\mathcal{M}}
\newcommand\NN{\mathcal{N}}

\renewcommand\d{\mathrm{d}}
\newcommand\dx{\d x}

\newcommand{\weakly}{\rightharpoonup}

\newcommand{\dualspace}{^\star}
\newcommand{\adjoint}{^\star}

\DeclareMathAlphabet{\mathpzc}{OT1}{pzc}{m}{it}

\newcommand\pBws{\partial_B^{ws}}
\newcommand\pBss{\partial_B^{ss}}
\newcommand\pBww{\partial_B^{ww}}
\newcommand\pBsw{\partial_B^{sw}}

\newcommand\capa{\operatorname{cap}}
\newcommand{\fsupp}{\operatorname{f-supp}}
\newcommand{\supp}{\operatorname{supp}}

\newcommand\toSOT{\stackrel{\textup{SOT}}{\longrightarrow}}
\newcommand\toWOT{\stackrel{\textup{WOT}}{\longrightarrow}}
\newcommand\togamma{\stackrel{\gamma}{\to}}
\newcommand\toGamma{\stackrel{\Gamma}{\to}}

\newtheorem{theorem}{Theorem}[section]
\newtheorem{lemma}[theorem]{Lemma}
\newtheorem{proposition}[theorem]{Proposition}

\newtheorem{corollary}[theorem]{Corollary}

\newtheorem{definition}[theorem]{Definition}

\crefname{assumption}{Assumption}{Assumptions}

\definecolor{darkgreen}{rgb}{0,0.5,0}
\definecolor{darkred}{rgb}{0.8,0,0}
\begin{document}
%fakesection: Title and co
\title{Generalized derivatives for the solution operator of the obstacle problem}

\author{%
	Anne-Therese Rauls%
	\footnote{%
		Technische Universität Darmstadt,
		Department of Mathematics,
		Research Group Optimization,
		64293 Darmstadt,
		Germany,
		\url{http://www3.mathematik.tu-darmstadt.de/hp/ag-optimierung/rauls-anne-therese/startseite.html},
		\email{rauls@mathematik.tu-darmstadt.de}%
		}
	\and
	Gerd Wachsmuth%
	\footnote{%
		Brandenburgische Technische Universität Cottbus-Senftenberg,
		Institute of Mathematics,
		Chair of Optimal Control,
		03046 Cottbus,
		Germany,
		\url{https://www.b-tu.de/fg-optimale-steuerung},
		\email{gerd.wachsmuth@b-tu.de}%
	}%
	\orcid{0000-0002-3098-1503}%
}
 % \date{\today}
\publishers{}
 %\dedication{}
\maketitle

\begin{abstract}
	We characterize generalized derivatives
	of the solution operator of the obstacle problem.
	This precise characterization requires the usage of
	the theory of so-called capacitary measures
	and the associated solution operators of
	relaxed Dirichlet problems.
	The generalized derivatives can be used
	to obtain a novel necessary optimality condition for the optimal control of the obstacle problem
	with control constraints.
	A comparison shows that this system is stronger than the known system of C-stationarity.
\end{abstract}
 
 
\begin{msc}
	\mscLink{49K40},
	\mscLink{47J20},
	\mscLink{49J52}, 
	\mscLink{58C20}
\end{msc}

\section{Introduction}
We consider the obstacle problem
\begin{equation}
	\label{eq:obstacle_problem}
	\text{Find } y \in K :
	\quad
	\dual{-\Delta y - u}{z - y} \ge 0
	\quad \forall z \in K.
\end{equation}
Here,
$\Omega \subset \R^d$ is a bounded, open set
and the closed, convex set $K \subset H_0^1(\Omega)$
is given by
\begin{equation*}
	K :=
	\{ v \in H_0^1(\Omega) \mid v \ge \psi \; \text{q.e.\ on } \Omega \},
\end{equation*}
where
$\psi : \Omega \to [-\infty, \infty)$ is a given quasi upper-semicontinuous function.
We assume $K \ne \emptyset$.
It is well known
that for each $u \in H^{-1}(\Omega)$
there exists a unique solution
$y := S(u) \in H_0^1(\Omega)$ of \eqref{eq:obstacle_problem}.
Moreover, the mapping $S : H^{-1}(\Omega) \to H_0^1(\Omega)$
is globally Lipschitz continuous.

Our main goal is the characterization of so-called generalized derivatives
of the mapping $S$.
That is, given $u \in H^{-1}(\Omega)$,
we are going to characterize the limit points of $S'(u_n)$,
where $\{u_n\}$ is a sequence of points in which $S$ is Gâteaux differentiable
and which converge towards $u$.
Since the involved spaces are infinite dimensional,
there is some choice concerning the topologies.
We will equip the space of operators with the weak or the strong operator topology
and on $H^{-1}(\Omega)$ we use the weak or strong topology.
When considering the weak topology on $H^{-1}(\Omega)$,
we also require that $\{S(u_n)\}$ converges weakly to $S(u)$.
At this point we also recall the famous result
\cite[Théorème~1.2]{Mignot1976}
which shows that $S$ is Gâteaux differentiable on a dense subset of $H^{-1}(\Omega)$.
Thus, each point $u \in H^{-1}(\Omega)$ can be approximated by
differentiability points of $S$.

The precise characterization of these generalized derivatives
will involve the notion of ``capacitary measures''
and ``relaxed Dirichlet problems''.
A comprehensive introduction to these topics will be given in \cref{Sec:capacitary_measures} below.
A Borel measure is a $\sigma$-additive set function on the Borel $\sigma$-algebra with values in $[0,\infty]$.
A capacitary measure $\mu$ is a Borel measure
which does not charge sets of capacity zero
and which satisfies a regularity condition, see \cref{def:capacitary_measure}.
For each capacitary measure $\mu$, we can consider the solution operator $u \mapsto y$ of
\begin{equation}
	\label{eq:relaxed_dirichlet_intro}
	-\Delta y + \mu \, y = u
\end{equation}
equipped with homogeneous Dirichlet boundary conditions,
see \cref{Sec:capacitary_measures} for the precise definition of this solution operator.
It is well known that the solution operator of \eqref{eq:relaxed_dirichlet_intro}
can be approximated (in the weak operator topology)
by the solution operators of
\begin{equation}
	\label{eq:dirichlet_intro}
	-\Delta y = u \text{ in } H^{-1}(\Omega_n), \qquad y \in H_0^1(\Omega_n)
\end{equation}
for some sequence of open sets $\Omega_n \subset \Omega$.
Moreover, each sequence of solution operators of \eqref{eq:dirichlet_intro}
converges (along a subsequence) to a solution operator of \eqref{eq:dirichlet_intro}
with an appropriate capacitary measure $\mu$.
This motivates to term \eqref{eq:relaxed_dirichlet_intro}
a relaxed Dirichlet problem.
Our analysis reveals that the generalized derivatives of $S$
are precisely sets of solution operators of \eqref{eq:relaxed_dirichlet_intro}
with appropriate conditions on $\mu$.

After we have established the characterization of the generalized derivatives,
we turn our attention to the optimal control of the obstacle problem
\begin{equation}
	\label{eq:optimal_control_obstacle}
	\text{Minimize} \quad J(y,u) \quad \text{with }  y = S(u)
	\text{ and } u \in \Uad.
\end{equation}
Here, $J : H_0^1(\Omega) \times L^2(\Omega) \to \R$
is assumed to be Fréchet differentiable with partial derivatives
$J_y$ and $J_u$,
and $\Uad \subset L^2(\Omega)$ is assumed to be closed and convex.
By a formal application of Lagrange duality, we arrive at
the stationarity system
\begin{equation}
	\label{eq:stat_system_intro}
	0 \in L\adjoint J_y(y,u) + J_u(y,u) + \NN_{\Uad}(u)
	\qquad
	\text{for some } L \in \partial_B S(u).
\end{equation}
Here, $\partial_B S(u)$ is a generalized differential of $S$ at $u$,
and $\NN_{\Uad(u)}$ is the normal cone in the sense of convex analysis of $\Uad$ at $u$.
We will see that
(for a certain choice of the involved topologies in the definition of $\partial_B S(u)$)
this system is slightly stronger than the so-called system of C-stationarity
from \cite{SchielaWachsmuth2013}.
Moreover, by inspecting the proof of
\cite{SchielaWachsmuth2013},
it is possible to strengthen this system of C-stationarity
such that it becomes equivalent to \eqref{eq:stat_system_intro}.
Therefore, our research leads to the discovery
of a new necessary optimality condition for \eqref{eq:optimal_control_obstacle}
which improves the known system of C-stationarity.

We put our work into perspective.
Our research was highly influenced by the recent contribution
\cite{ChristofClasonMeyerWalther2017}.
Therein, the authors considered the non-smooth partial differential equation
\begin{equation}
	\label{eq:nonsmooth_pde}
	-\Delta y + \max\{y, 0\} = u
\end{equation}
equipped with homogeneous Dirichlet boundary conditions.
They characterized generalized derivatives for the solution operator
mapping $u \mapsto y$.
Subsequently, these generalized derivatives are used to
derive and compare optimality conditions for the optimal control of \eqref{eq:nonsmooth_pde}.
Furthermore,
a single generalized gradient for the infinite-dimensional obstacle problem
was computed in \cite{RaulsUlbrich2018}.
This gradient is contained in all of the generalized derivatives
that we will consider and the approach gives a hint
how the generalized differential involving strong topologies
might look like.
The derivation there uses different tools
and while being able to treat also the variational inequality
\begin{equation*}
	\text{Find } y \in K :
	\quad
	\dual{-\Delta y - f(u)}{z - y} \ge 0
	\quad \forall z \in K.
\end{equation*}
for an appropriate monotone operator $f$ with range smaller than $H^{-1}(\Omega)$,
it is hard to characterize the entire generalized differential
involving strong topologies with this approach,
let alone those involving also weak topologies.
We are not aware of any other contribution
in which generalized derivatives of nonsmooth infinite-dimensional mappings are computed.
There is, however, a vast amount of literature
in the finite-dimensional setting.
We only mention
\cite{KlatteKummer2002,OutrataKocvaraZowe1998}.

Let us give an outline of this work. 
In the following section, 
we recall the relevant notions and results from capacity theory (\cref{subsec:capacity_theory}), 
recapitulate differentiability properties of the obstacle problem (\cref{subsec:known_results_obstacle}) and introduce the generalized differentials we are dealing with in this paper (\cref{subsec:generalized_differentials}). 
We review the concepts of capacitary measures, relaxed Dirichlet problems
and $\gamma$-convergence 
in \cref{Sec:capacitary_measures}. 
The generalized differentials
of the solution operator
to the obstacle problem
associated to the strong operator topology
will be established in \cref{sec:generalized_derivatives_SOT}.
Under additional regularity assumptions 
we characterize the generalized differential
involving the strong topology in $H^{-1}(\Omega)$
and the weak operator topology for the operators 
in \cref{sec:strong_weak_generalized_derivative}.
In \cref{sec:weak_weak_generalized_derivative}, 
we give an example to show 
that the generalized differential 
involving only weak topologies
can be very large,
even in points of differentiability.
Based on the developed characterizations of generalized derivatives,
we discuss stationarity systems
for the optimal control
of the obstacle problem
with control constraints
in \cref{sec:stationarity_systems}.

\section{Notation and known results}
\label{sec:notation_preliminaries}
In this work, $\Omega \subset \R^d$ is an open bounded set
in dimension $d \ge 2$.
By $H_0^1(\Omega)$,
we denote the usual Sobolev space.
Its norm is given by
$\norm{u}_{H_0^1(\Omega)}^2 = \int_\Omega \abs{\nabla u}^2 \, \dx$
and
the duality pairing between $H^{-1}(\Omega) := H_0^1(\Omega)\dualspace$
and $H_0^1(\Omega)$
is $\dual{\cdot}{\cdot}$.

We often deal with subsets of $\Omega$
that are defined only up to a set of capacity zero,
see also \cref{subsec:capacity_theory}.
As a consequence,
relations between such sets,
such as inclusions and equalities,
are meaningful only up to a set of capacity zero.
For subsets $B, C$ that are defined up to capacity zero,
we distinguish such relations by writing $B\subset_q C$, $B\supset_q C$ or $B=_q C$.
Similarly,
definitions of sets up to capacity zero,
such as the zero set of a family of quasi-continuous representatives,
see \cref{subsec:capacity_theory},
are denoted by ``$:=_q$''.

\subsection{Introduction to capacity theory}
\label{subsec:capacity_theory}
%\alert{GW: Ich würde vorschlagen, alles bis \eqref{characterization_strictly_active_set} in
%	\cref{sec:notation_preliminaries} vorzuverlegen. Das braucht man ja dort schon für
%	die Diskussion des Obstacle-Problems.
%	AT: Finde ich gut, habe es nur erstmal so gelassen weil es in dem stichpunktartigen Aufbau vorher so geordnet war.}
We collect some fundamentals on capacity theory. 
For the definitions, see e.g.\ \cite[Sections~5.8.2, 5.8.3]{AttouchButtazzoMichaille2014}, \cite[Definition~6.2]{DelfourZolesio2011} or \cite[Definition~6.47]{BonnansShapiro2000}.

\begin{definition}
	\begin{enumerate}
		\item 
		For every set $A \subset \Omega$
		the capacity
		(in the sense of $H_0^1(\Omega)$)
		is defined as
		\begin{equation*}
		\capa(A):=\inf \{\norm{u}_{H^1_0(\Omega)}^2: 
		u \in H^1_0(\Omega), 
		u\geq 1 \text{ a.e.\ in a neighborhood of } A\}.
		\end{equation*}
		\item 
		A subset
		$\hat{\Omega}\subset \Omega$ is called quasi-open
		if for all $\varepsilon>0$
		there is an open set $O_\varepsilon \subset \Omega$
		with $\capa(O_\varepsilon)<\varepsilon$
		such that $\hat{\Omega} \cup O_\varepsilon$ is open.
		The relative complement of a quasi-open set in $\Omega$ is called quasi-closed.
		\item
		A function $v \colon \Omega \to \overline\R = [-\infty,+\infty]$ is called quasi-continuous
		(quasi lower-semi\-continuous, quasi upper-semicontinuous, respectively)
		if for all $\varepsilon>0$
		there is an open set $O_\varepsilon\subset \Omega$
		with $\capa(O_\varepsilon)<\varepsilon$
		such that $v$ is continuous
		(lower-semicontinuous, upper-semicontinuous, respectively)
		on $\Omega \setminus O_\varepsilon$.
	\end{enumerate}
\end{definition}

If a property holds on $\Omega$ except on a set of zero capacity,
we say that this property holds quasi-everywhere (q.e.) in $\Omega$.
It is well known that each $v \in H_0^1(\Omega)$ possesses a quasi-continuous representative,
which is uniquely determined up to values on a set of zero capacity,
see e.g.\ \cite[Lemma~6.50]{BonnansShapiro2000}
or \cite[Chapter~8, Theorem~6.1]{DelfourZolesio2011}.
%The active set $A(u)=\{S(u)=\psi\}$ has to be understood in this sense.
%It is defined up to a set of zero capacity and quasi-closed.
Moreover, the proof in the former reference yields that this representative
can be chosen to be even Borel measurable.
From now on, we will always use quasi-continuous and Borel measurable representatives
when working with functions from $H_0^1(\Omega)$.

Similarly, every quasi lower-/upper-semicontinuous function can be made Borel measurable
by a modification on a set of capacity zero.
Indeed, for a quasi upper-semicontinuous function $\psi$, the sets $\{\psi < q\}$
are quasi-open for all $q \in \Q$.
Hence, there are Borel sets $O_q$ of capacity zero, such that
$\{\psi < q\} \cup O_q$ is a Borel set for each $q \in \Q$.
% , see, e.g., \cite[Appendix~A]{Wachsmuth2014}.
% There, we can choose $O_q := \bigcap_i G_{1/i}$.
By setting $\psi$ to $-\infty$ on $\bigcup_{q \in \Q} O_q$,
the function is still quasi upper-semicontinuous and becomes Borel measurable.
W.l.o.g., we will assume that the obstacle $\psi$ is Borel measurable.

\begin{lemma}
	\label{lem:pointwise_quasi_everywhere_convergence_of_a_subsequence}
	Let $\{v_n\} \subset H_0^1(\Omega)$, $v \in H_0^1(\Omega)$
	and assume that $v_n \to v$ in $H_0^1(\Omega)$.
	Then there is a subsequence of $\{v_n\}$,
	such that the sequence (of quasi-continuous representatives of) $\{v_n\}$ converges pointwise quasi-everywhere to (the quasi-continuous representative of) $v$.
\end{lemma}

\begin{proof}
	See \cite[Lemma~6.52]{BonnansShapiro2000}.
\end{proof}

It is possible to extend the Sobolev space $H_0^1(\Omega)$
to quasi-open subsets $\hat{\Omega} \subset \Omega$
by setting 
\begin{equation*}
H_0^1(\hat{\Omega}):=
\{v \in H_0^1(\Omega): v=0 \text{ q.e.\ on } \Omega \setminus \hat{\Omega}\}.
\end{equation*}
We point out that this definition is consistent with the usual definition
of $H_0^1(\hat\Omega)$ in the case that $\hat\Omega$ is open,
see \cite[Theorem~4.5]{HeinonenKilpelaeinenMartio1993}.
By \cref{lem:pointwise_quasi_everywhere_convergence_of_a_subsequence}, 
the space $H_0^1(\hat\Omega)$ is a closed subspace of $H_0^1(\Omega)$.

\begin{lemma}
	\label{lem:quasi-covering}
	Let $\hat{\Omega} \subset \Omega$ be quasi-open
	and assume there is a sequence of quasi-open sets $\{\hat\Omega_n\}$
	such that $\{\hat{\Omega}_n\}$ is increasing in $n$
	and such that $\hat{\Omega}=_q\bigcup_{n=1}^\infty \hat{\Omega}_n$.
	Let $v \in H_0^1(\hat{\Omega})$. 
	Then there is a sequence $\{v_n\}$ 
	with $v_n \in H_0^1(\hat{\Omega}_n)$ for each $n \in \mathbb{N}$
	such that $v_n \to v$ in $H_0^1(\Omega)$.
	Furthermore, 
	it holds $\sup|v_n|\leq\sup|v|$.
\end{lemma}

\begin{proof}
	The sequence $\{\hat{\Omega}_n\}$ represents a quasi-covering of $\hat{\Omega}$,
	therefore,
	combining \cite[Theorem~2.10 and Lemma~2.4]{KilpelaeinenMaly1992},
	we find a sequence $\{v_n\}$ 
	such that $v_n \to v$ in $H_0^1(\Omega)$
	and such that each $v_n$ is a finite sum of elements in $\bigcup_{m=1}^\infty H_0^1(\hat{\Omega}_m)$.
	Furthermore,
	$\sup|v_n|\leq\sup|v|$.
	Since the sets $\hat{\Omega}_n$ are increasing,
	for each $n \in \mathbb{N}$ there is $j \in \mathbb{N}$ such that $v_n \in \bigcap_{m=j}^\infty H_0^1(\hat{\Omega}_m)$.
	We extend the sequence by adding copies of elements in $\{v_n\}$ to the original sequence.
	This yields a sequence with the desired properties.
\end{proof}

Using the same ideas,
we can characterize the sum of two Sobolev spaces on quasi-open domains.
\begin{lemma}
	\label{lem:union_of_domains}
	Let $\Omega_1, \Omega_2 \subset \Omega$ be quasi-open. Then,
	\begin{equation*}
		\overline{ H_0^1(\Omega_1) + H_0^1(\Omega_2) }^{H_0^1(\Omega)}
		=
		H_0^1(\Omega_1 \cup \Omega_2).
	\end{equation*}
	Moreover,
	for every $v \in H_0^1(\Omega_1 \cup \Omega_2)^+$,
	there exist sequences
	$\{v^{(1)}_n\} \subset H_0^1(\Omega_1)^+$
	and
	$\{v^{(2)}_n\} \subset H_0^1(\Omega_2)^+$
	with
	$0 \le v^{(1)}_n + v^{(2)}_n \le v$ q.e.\ on $\Omega$ for all $n \in \N$,
	and
	$v^{(1)}_n + v^{(2)}_n \to v$ in $H_0^1(\Omega)$.
\end{lemma}
\begin{proof}
	Since $\{\Omega_1, \Omega_2\}$
	is a quasi-covering of $\Omega_1 \cup \Omega_2$,
	we can argue as in the proof of \cref{lem:quasi-covering}
	to obtain the first identity.
	For the second assertion,
	an inspection of the proofs of
	\cite[Theorem~2.10 and Lemma~2.4]{KilpelaeinenMaly1992}
	shows that the approximating functions can be chosen to be
	pointwise bounded by $0$ and $v$.
\end{proof}

We also recall that positive elements 
in the dual space $H^{-1}(\Omega)$ of $H_0^1(\Omega)$
can be identified with regular Borel measures which are finite on compact sets.
Here, a Borel measure on $\Omega$ is a measure over the Borel $\sigma$-algebra $\mathcal{B}$,
which is the smallest $\sigma$-algebra containing all open subsets of $\Omega$.
We call a Borel measure $\mu$ regular
if 
\begin{equation*}
\mu(B)=\inf\{\mu(O): B \subset O, O \text{ is open}\}=\sup\{\mu(C): C \subset B, C \text{ is compact}\}
\end{equation*}
holds for all $B \in \mathcal{B}$.
Finally, $\mu$ is said to be finite on compact sets, if $\mu(K) < +\infty$
for all compact subsets $K \subset \Omega$.
\begin{lemma}
	\label{lem:fine_support}
	Let $\xi \in H^{-1}(\Omega)^+$ be given, i.e., 
	$\dual{\xi}{v} \geq 0$ for all $v \in H_0^1(\Omega)$ with $v \geq 0$.
	\begin{enumerate}
		\item 
		The functional
		$\xi$ can be identified with a regular Borel measure on $\Omega$
		which is finite on compact sets
		and which possesses
		the following property:
		For every Borel set $B \subset \Omega$
		with $\capa(B) = 0$,
		we have $\xi(B)=0$.
		\item 
		Every function $v \in H_0^1(\Omega)$ is $\xi$-integrable
		and it holds
		\begin{equation*}
		\dual{\xi}{v}_{H^{-1}(\Omega), H_0^1(\Omega)}=\int_{\Omega} v\,\d\xi.
		\end{equation*}
		%In particular, $v$ is measurable.
%		\alert{GW: Die Messbarkeit ist doch sowieso klar, oder? $\xi$ ist ein Borelmaß und
%			jede Lebesgue-messbare Funktion hat einen Repräsentanten, der Borel-messbar ist.
%			AT: Ja..}
		\item 
		There is a quasi-closed set $\fsupp(\xi) \subset \Omega$ 
		with the property that
		for all $v \in H_0^1(\Omega)^+$
		it holds
		$\dual{\xi}{v}_{H^{-1}(\Omega),H_0^1(\Omega)}=0$
		if and only if 
		$v=0$ q.e.\ on $\fsupp(\xi)$.
		The set $\fsupp(\xi)$ is uniquely defined up to a set of zero capacity.
	\end{enumerate}
\end{lemma}

Proofs for statement (i) and (ii)
can be found in \cite[p.\ 564, 565]{BonnansShapiro2000}.
Note that the regularity of $\mu$ is implied by the property of being finite on compact sets,
see \cite[Theorem~2.18]{Rudin1987}.
For part (iii) we refer to \cite[Lemma~3.7]{HarderWachsmuth2017:2}.
See also \cite[Lemma~3.5]{HarderWachsmuth2017:2} and
for a different description of the fine support $\fsupp(\xi)$ in (iii)
see \cite[Lemma~A.4]{Wachsmuth2014}.

%In \cite[Appendix~A]{Wachsmuth2014},
%it is shown that the strictly active set $A_s(u)$
%with respect to $u \in H^{-1}(\Omega)$
%can be identified with the fine support 
%of the multiplier $\xi=-\Delta S(u)-u$,
%i.e.,
%\begin{equation}
%\label{characterization_strictly_active_set}
%A_s(u)=\fsupp(\xi).
%\end{equation}

\subsection{Differentiability of the solution operator of the obstacle problem}
\label{subsec:known_results_obstacle}
For the variational inequality \eqref{eq:obstacle_problem},
we consider the solution operator $S \colon H^{-1}(\Omega) \to H_0^1(\Omega)$
that maps $u \in H^{-1}(\Omega)$ to the unique solution $y=S(u)$ of \eqref{eq:obstacle_problem}.
We define the active set associated with $u \in H^{-1}(\Omega)$ by
\begin{equation*}
	A(u):=_q\{\omega \in \Omega: S(u)(\omega)=\psi(\omega)\}
\end{equation*}
and the inactive set by
\begin{equation*}
	I(u):=_q\Omega \setminus A(u).
\end{equation*}
We emphasize that these sets are defined up to sets of capacity zero
since we always work with the quasi-continuous
representatives of functions from $H_0^1(\Omega)$,
see also \cref{subsec:capacity_theory} above.
Furthermore, $A(u)$ is quasi-closed,
$I(u)$ is quasi-open
and both sets are Borel measurable.

It is well known that $S$ is directionally differentiable
and that the directional derivative
at $u \in H^{-1}(\Omega)$ in direction $h \in H^{-1}(\Omega)$,
which is denoted by $S'(u;h)$,
solves the variational inequality
\begin{equation}
\label{eq:directional_derivative}
	\text{Find } y \in \mathcal{K}(u) :
	\quad
	\dual{-\Delta y - h}{z - y} \ge 0
	\quad \forall z \in \mathcal{K}(u),
\end{equation}
see \cite[]{Mignot1976}.
Here,
$\mathcal{K}(u)$ denotes the critical cone,
which, according to \cite[Lemma~3.1]{Wachsmuth2014}, has the following structure:
\begin{equation*}
	\mathcal{K}(u):=\{ z \in H_0^1(\Omega):
	z\geq 0 \text{ q.e.\ on } A(u) 
	\text{ and } z=0 \text{ q.e.\ on } A_s(u)\}.
\end{equation*} 
Here, the strictly active set $A_s(u)$ is a quasi-closed subset of the active set $A(u)$.
It has a representation
in terms of the fine support of the multiplier $\xi=-\Delta S(u)-u \in H^{-1}(\Omega)^+$,
see \cite[Appendix~A]{Wachsmuth2014}.
In fact, it holds
\begin{equation}
\label{characterization_strictly_active_set}
	A_s(u)=_q\fsupp(\xi).
\end{equation}
Again, we emphasize that this definition is unique up to a subset of capacity zero.

The following lemma characterizes the points 
in which $S$ is G\^ateaux differentiable.

\begin{lemma}
	\label{lem:gateaux_S}
	The solution operator $S$ of the obstacle problem \eqref{eq:obstacle_problem}
	is G\^ateaux differentiable in $u \in H^{-1}(\Omega)$
	if and only if
	the strict complementarity condition is valid in $u$,
	i.e.,
	if and only if the equality $A(u)=_q A_s(u)$ holds.
\end{lemma}

\begin{proof}
	Assume that $A(u)=_q A_s(u)$ holds.
	Then $\mathcal{K}(u)=\{z \in H_0^1(\Omega): z=0 \text{ q.e.\ on } A(u)\}$ is a linear subspace
	and the variational inequality \eqref{eq:directional_derivative}
	for the directional derivative $S'(u;h)$ reduces to
	\begin{equation*}	
		\text{Find } y \in \mathcal{K}(u) :
		\quad
		\dual{-\Delta y - h}{z} = 0
		\quad \forall z \in \mathcal{K}(u),
	\end{equation*}
	i.e., $S'(u;\cdot)$ is linear and bounded.

	For the reverse implication,
	assume that $S$ is G\^ateaux differentiable in $u \in H^{-1}(\Omega)$.
	By the variational inequality \eqref{eq:directional_derivative}
	we obtain
	that the image of $S'(u;\cdot)$ is contained in $\mathcal{K}(u)$.
	Conversely, let $v \in \mathcal{K}(u)$ be arbitrary.
	Then we can check $v=S'(u;-\Delta v)$,
	which implies that $\mathcal{K}(u)$ 
	coincides with the image of $S'(u;\cdot)$.
	Thus, $\mathcal{K}(u)$ is a linear subspace of $H_0^1(\Omega)$.
	Using $A_s(u) \subset_q A(u)$, we trivially have
	\begin{equation*}
		H_0^1\bigh(){ \Omega \setminus A(u) }
		\subset
		\KK(u)
		\subset
		H_0^1\bigh(){ \Omega \setminus A_s(u) }
		.
	\end{equation*}
	Now, for $v \in H_0^1\bigh(){ \Omega \setminus A_s(u)}^+$,
	we have $v \ge 0$ q.e.\ in $\Omega$.
	Hence, $v \in \KK(u)$ and, since $\KK(u)$ is a subspace,
	we also have $-v \in \KK(u)$.
	This leads to $v \ge 0$ and $v \le 0$ q.e.\ on $A(u)$.
	Therefore, $v \in H_0^1\bigh(){ \Omega \setminus A(u) }$.
	This shows
	\begin{equation*}
		H_0^1\bigh(){ \Omega \setminus A(u) }
		=
		\KK(u)
		=
		H_0^1\bigh(){ \Omega \setminus A_s(u) }
		.
	\end{equation*}
	Finally,
	\cref{thm:torsion_function_properties}
	below
	implies that the equality $A(u)=_q A_s(u)$
	holds.
\end{proof}

To summarize,
the G\^ateaux derivative 
of the solution operator of the obstacle problem
in differentiability points $u \in H^{-1}(\Omega)$
is given by the operator $L_{I(u)} \in \LL(H^{-1}(\Omega),H_0^1(\Omega))$,
where for $h \in H^{-1}(\Omega)$,
the element $L_{I(u)}(h)$ 
is the solution to the boundary value problem
\begin{equation}
\label{G\^ateaux_derivatives}
	y \in H_0^1(I(u)): \quad -\Delta y=h.
\end{equation}
This equality has to be understood in the sense of $H_0^1(I(u))\dualspace$,
i.e.,
$\dual{-\Delta y - h}{v} = 0$ for all $v \in H_0^1(I(u))$.

\subsection{Generalized differentials}
\label{subsec:generalized_differentials}
The generalized differentials, which we will consider,
consist of operators in $\LL(X,Y)$.
In their definition,
we will differentiate between different topologies
on $X$ and $\LL(X,Y)$.
We consider the following standard operator topologies on $\LL(X,Y)$.

\begin{definition}
	Let $X$ and $Y$ be Banach spaces and $\{L_n\},L \subset \LL(X,Y)$.
	\begin{enumerate}
		\item 
		We say that the sequence $\{L_n\}$ converges to $L$ in the strong operator topology (SOT)
		if and only if 
		$\{L_n h\}$ converges to $Lh$ in $Y$ for all $h \in X$.
		If $\{L_n\}$ converges to $L$ in the strong operator topology,
		we write $L_n \toSOT L$.
		\item 
		We say that the sequence $\{L_n\}$ converges to $L$ in the weak operator topology (WOT)
		if and only if 
		$\{L_n h\}$ converges to $Lh$ weakly in $Y$ for all $h \in X$.
		If $\{L_n\}$ converges to $L$ in the weak operator topology,
		we write $L_n \toWOT L$.
	\end{enumerate}
\end{definition}

From the uniform boundedness principle,
we obtain that a sequence of operators
which converges in WOT
has to be bounded.
\begin{lemma}
	\label{lem:norm_bounded_WOT}
	Let $\{L_n\} \subset \LL(X,Y)$ 
	and assume that $L_n \toWOT L$ 
	for some $L \in \LL(X,Y)$.
	Then there is a constant $C>0$
	such that $\|L_n\|_{\LL(X,Y)}\leq C$ for all $n \in \mathbb{N}$.
\end{lemma}
%\alert{GW: Ich habe den Beweis auskommentiert, er ist ja recht Standard.
%	AT: Okay, können auch gerne an anderen Stellen noch was kürzen, wenn wir feststellen, dass manche Argumente zu lang sind/ nicht unbedingt nötig sind.}
% \begin{proof}
% 	The convergence of $\{A_n\}$ to $A$ in the WOT implies that
% 	\begin{equation*}
% 		\sup_{n \in \mathbb{N}}|\dual{y^*}{A_n x}|< \infty
% 	\end{equation*}
% 	for all $x \in X$ and for all $y^* \in Y^*$.
% 	Fix $x \in X$.
% 	Define $\varphi_n \in Y^{**}$ by
% 	\begin{equation*}
% 		\dual{\varphi_n}{y^*}_{Y^{**},Y^*}
% 		=\dual{y^*}{A_n x}_{Y^*,Y}
% 	\end{equation*}
% 	for $y^* \in Y^*$.
% 	By the uniform boundedness principle,
% 	it holds $\|\varphi_n\|_{Y^{**}} \leq c_x$ for some constant $c_x>0$ and for all $n \in \mathbb{N}$.
% 	Since $\|\varphi_n\|_{Y^{**}}=\|A_n x\|_Y$,
% 	it follows that $\|A_n x\| \leq c_x$ for all $n \in \mathbb{N}$.
% 	Applying the uniform boundedness principle once more,
% 	we deduce the existence of $C>0$ with $\|A_n\|_{\LL(X,Y)} \leq C$.
% \end{proof}

The next lemma shows under which conditions
a product $L_n \, h_n$ converges.
\begin{lemma}
	\label{lem:SOT_und_WOT}
	Let $\{L_n\} \subset \LL(X,Y)$ and $\{h_n\} \subset X$ be sequences.
	\begin{enumerate}
		\item
			If $L_n \toSOT L$ and $h_n \to h$ then $L_n h_n \to L h$.
		\item
			If $L_n \toWOT L$ and $h_n \to h$ then $L_n h_n \weakly L h$.
		\item
			If $L_n \toWOT L$, $L_n\adjoint \toSOT L\adjoint$ and $h_n \weakly h$ then $L_n h_n \weakly L h$.
	\end{enumerate}
\end{lemma}
\begin{proof}
	In any case, the norm of the operators $L_n$ is uniformly bounded,
	see \cref{lem:norm_bounded_WOT}.
	Now, we use the identity
	\begin{equation*}
		L_n h_n - L h
		=
		(L_n h - L h) + L_n (h_n - h)
		.
	\end{equation*}
	In cases (i) and (ii), the claim follows immediately.

	In case (iii), $L_n h - L h \weakly 0$ is clear.
	To prove the weak convergence of the second addend, we take $f \in Y\dualspace$
	and have
	\begin{equation*}
		\dual{ f }{ L_n \, (h_n - h)}_{Y\dualspace,Y}
		=
		\dual{ L_n\adjoint f}{ h_n - h}_{X\dualspace,X}
		\to
		0
	\end{equation*}
	since $L_n\adjoint \toSOT L\adjoint$ by assumption.
\end{proof}

Now we define the generalized derivatives 
that we will deal with in this paper.

\begin{definition}
	\label{def:generalized_derivatives}
	Let $T \colon X \to Y$ be a locally Lipschitz mapping 
	from a separable Banach space $X$ 
	to a separable and reflexive Banach space $Y$.
	We denote the set of points in $X$ 
	in which $T$ is G\^ateaux differentiable by $D_T$.
	For $x \in X$ 
	we define the following generalized derivatives
	\begin{align*}
		\pBss T(x):=
		\{L \in \LL(X,Y):
		\exists \{x_n\} \subset D_T:
		x_n \to x \text{ in } X ,\;
		&T'(x_n) \toSOT L \text{ in } \LL(X,Y)\},
		\\
		\pBsw T(x):=
		\{L \in \LL(X,Y):
		\exists \{x_n\} \subset D_T:
		x_n \to x \text{ in } X ,\;
		&T'(x_n) \toWOT L \text{ in } \LL(X,Y)\},
		\\
		\pBws T(x):=
		\{L \in \LL(X,Y):
		\exists \{x_n\} \subset D_T:
		x_n \weakly x \text{ in } X ,\; 
		&T(x_n) \weakly T(x) \text{ in } Y,\\
		&T'(x_n) \toSOT L \text{ in } \LL(X,Y)\},
		\\
		\pBww T(x):=
		\{L \in \LL(X,Y):
		\exists \{x_n\} \subset D_T :
		x_n \weakly x \text{ in } X ,\; 
		&T(x_n) \weakly T(x) \text{ in } Y,\\
		&T'(x_n) \toWOT L \text{ in } \LL(X,Y)\}.
	\end{align*}
	Note that the first superscript refers to the mode of convergence
	of the points $x_n$ in $X$,
	whereas the second superscript refers to the type of operator topology
	for the convergence of $T'(x_n)$.
\end{definition}
In the literature, these generalized differentials
are sometimes called ``subderivatives''.
However, this notion is only senseful
for functions mapping into $\R$ (or, more generally, into an ordered set).

Note that, 
in contrast to \cite[Definition~3.1]{ChristofClasonMeyerWalther2017},
we also require that the values $\{T(x_n)\}$ converge weakly to $T(x)$
when considering the generalized differentials $\pBws T(x)$ and $\pBww T(x)$.
Since the solution operator $\tilde{S}$
to the non-smooth semilinear equation treated in \cite{ChristofClasonMeyerWalther2017}
is weakly (sequentially) continuous on the considered spaces, 
see \cite[Corollary~3.7]{ChristofClasonMeyerWalther2017},
it always fulfills 
$\tilde{S}(u_n) \weakly \tilde{S}(u)$
whenever $u_n \weakly u$,
anyway.
However, the solution operator $S$ of the obstacle problem
is not weakly (sequentially) continuous from $H^{-1}(\Omega)$ to $H_0^1(\Omega)$.

We collect some simple properties of the generalized derivatives.

\begin{proposition}
	\label{prop:properties_derivatives}
	Let $T \colon X \to Y$ be a globally Lipschitz continuous map 
	from a separable Banach space $X$ 
	to a separable, reflexive Banach space $Y$.
	\begin{enumerate}
		\item 
		For all $x \in X$ it holds
		\begin{equation*}
			\pBss T(x) 
			\subset \pBsw T(x) 
			\subset \pBww T(x) 
			\qquad\text{and}\qquad
			\pBss T(x) 
			\subset \pBws T(x) 
			\subset \pBww T(x).
		\end{equation*}
		\item 
		Let $x \in X$. 
		Suppose there is a sequence $\{x_n\} \subset X$ 
		with $x_n \to x$ in $X$ 
		and a sequence $\{L_n\} \subset \LL(X,Y)$ 
		with $L_n \in \pBss T(x_n)$ for all $n \in \mathbb{N}$. 
		Furthermore, 
		assume that $L_n \toSOT L$ for some $L \in \LL(X,Y)$. 
		Then $L$ is in $\pBss T(x)$.
		\item 
		Let $x \in X$. 
		Suppose there is a sequence $\{x_n\} \subset X$ 
		with $x_n \to x$ in $X$ 
		and a sequence $\{L_n\} \subset \LL(X,Y)$ 
		with $L_n \in \pBsw T(x_n)$ for all $n \in \mathbb{N}$. 
		Furthermore, 
		assume that $L_n \toWOT L$ 
		for some $L \in \LL(X,Y)$. 
		Then $L$ is in $\pBsw T(x)$.
	\end{enumerate}
\end{proposition}

\begin{proof}
	The assertion in (i) follows easily by the relation between the respective topologies. 
	(ii) can be found in \cite[Proposition~3.4]{ChristofClasonMeyerWalther2017}, 
	one just hast to replace $L^2(\Omega)$ 
	by an arbitrary separable Banach space $X$.
	We prove part (iii) similarly to \cite[Proposition~3.4]{ChristofClasonMeyerWalther2017}
	with the obvious modifications:
	Since $L_n \in \pBsw T(x_n)$,
	there are sequences
	$\big\{x_m^{(n)}\big\} \subset D_T$ 
	with $x_m^{(n)} \to x_n$ as $m \to \infty$
	and $T'\big(x_m^{(n)}\big) \toWOT L_n$ as $m \to \infty$.
	Since $X$ is separable
	and since the properties of $Y$ imply that $Y\dualspace$ is separable as well,
	we can find sequences $\{h_n\}$ and $\{y_n\dualspace\}$
	that are dense in $X$, respectively $Y\dualspace$.
	For all $n \in \mathbb{N}$
	fix $m(n) \in \mathbb{N}$ with
	\begin{align*}
		\big|\bigdual{y_l\dualspace}{T'\big(x_{m(n)}^{(n)}\big)h_k-L_n h_k}\big|&<1/n 
		\qquad \forall k,l=1,\ldots,n,
		\\
		\big\|x_{m(n)}^{(n)}-x_n\big\|&\leq 1/n.
	\end{align*}
	For fixed $h \in X$, $y\dualspace \in Y\dualspace$, and for all $n \in \mathbb{N}$ we define
	\begin{align*}
		\bar{h}_n & :=\arg\min\{\|h_k-h\|_X: 1\leq k \leq n\}, \\
		\bar{y}_n\dualspace & :=\arg\min\{\|y_k\dualspace-y\dualspace\|_{Y\dualspace}: 1\leq k \leq n\}.
	\end{align*}
	These definitions imply that
	$\bar{h}_n \to h$ in $X$ 
	and $\bar{y}_n\dualspace \to y\dualspace$ in $Y\dualspace$.
	We mention that all elements in $\pBww T(x)$
	are bounded by the Lipschitz constant of $T$,
	see
	\cite[Lemma~3.2(iii)]{ChristofClasonMeyerWalther2017}.
	In particular,
	$\big\|T'\big(x_{m(n)}^{(n)}\big)-L_n\big\|_{\mathcal{L}(X,Y)}$ is bounded.
	This shows
	\begin{align*}
		&\big|\bigdual{y\dualspace}{T'\big(x_{m(n)}^{(n)}\big)h-L h}\big|
		\\
		&\qquad\leq\big|\bigdual{\bar{y}_n\dualspace}{T'\big(x_{m(n)}^{(n)}\big)\bar{h}_n-L_n\bar{h}_n}\big|
		+\big|\bigdual{\bar{y}_n\dualspace-y\dualspace}{T\big(x_{m(n)}^{(n)}\big)\bar{h}_n-L_n\bar{h}_n}\big|\\
		&\qquad\qquad+\big|\bigdual{y\dualspace}{(T'\big(x_{m(n)}^{(n)}\big)-L_n)(\bar{h}_n-h)}\big|
		+|\dual{y\dualspace}{L_n h-L h}|\\
		 &\qquad\leq 1/n+\|\bar{y}_n\dualspace-y\dualspace\|_{Y\dualspace}\big\|T\big(x_{m(n)}^{(n)}\big)-L_n\big\|_{\mathcal{L}(X,Y)}\|\bar{h}_n\|_X\\
		&\qquad\qquad+\|y\dualspace\|_{Y\dualspace}\big\|T'\big(x_{m(n)}^{(n)}\big)-L_n\big\|_{\mathcal{L}(X,Y)}\|\bar{h}_n-h\|_X+|\dual{y\dualspace}{L_n h-L h}|
		\to 0
	\end{align*}
	as $n \to \infty$.
	Together with $x_{m(n)}^{(n)} \to x$,
	we obtain the desired $L \in \pBsw T(x)$.
\end{proof}

\section{Introduction to capacitary measures}
\label{Sec:capacitary_measures}
The goal of this paper is the characterization
of generalized derivatives of the solution operator $S$.
In \cref{lem:gateaux_S}, we have seen that
$S'(u; \cdot)$
is of the form
$L_{I(u)}$
for all differentiability points $u \in D_S$,
see also \eqref{G\^ateaux_derivatives}.
In the definitions of the generalized derivatives
limits (in WOT or SOT) of such solution operators $L_{I(u)}$ appear,
see \cref{def:generalized_derivatives}.
Hence, we need to know which operators in $\LL( H^{-1}(\Omega), H_0^1(\Omega) )$
can appear as limits (in WOT or SOT)
of sequences of solution operators $L_{I(u)}$.

We will see that this question can be adequately answered
by the concept of so-called capacitary measures.
For the convenience of the reader,
we will give a self-contained introduction.
We suggest
\cite[Section~4.3]{BucurButtazzo2005}, \cite[Section~2]{DalMasoGarroni1994} or \cite[Section~2.2]{DalMasoMurat2004}
for further material.

We also remark that
\cref{lem:h01_weak_convergence_and_norms}, \cref{thm:gamma_limit_of_sum_of_measures} and the second half of \cref{thm:torsion_function_properties} are new results,
while the remaining results can be found in the mentioned references
or are easy corollaries of existing results in the literature.

\begin{definition}
	\label{def:capacitary_measure}
	Let $\MM_0(\Omega)$ be the set of all Borel measures $\mu$ on $\Omega$
	such that $\mu(B)=0$ for every Borel set $B \subset \Omega$ with $\capa(B)=0$
	and such that $\mu$ is regular 
	in the sense that
	$\mu(B)=\inf\{\mu(O): O \text{ quasi-open, } B\subset_q O\}$.
\end{definition}

The set $\MM_0(\Omega)$ is called 
the set of \emph{capacitary measures} on $\Omega$.
The name stems from the fact that,
on the one hand,
$\mu(B)=0$ for all Borel sets $B \subset \Omega$ with $\capa(B)=0$,
and on the other hand,
$\mu(B)=0$ for all $\mu \in \MM_0(\Omega)$ implies that $\capa(B)=0$,
see \cite[Lemma~6.55]{BonnansShapiro2000}.

Recall that we work with Borel measurable representatives,
that is, $v \in H_0^1(\Omega)$
is always assumed to be quasi-continuous and Borel measurable.
Since $\mu \in \MM_0(\Omega)$ is a Borel measure,
$v$ is $\mu$-measurable. Further,
for $p \in [1,\infty)$, we can define the integral
\begin{equation*}
	\int_\Omega \abs{v}^p \, \d\mu \in [0,\infty]
\end{equation*}
in the usual way.
In the case that the integral is finite, we write $v \in L^p_\mu(\Omega)$.
Note that this integral does not depend on the actual representative
of $v$, since the quasi-continuous representatives differ only on
sets of capacity zero whereas $\mu$ vanishes on sets of capacity zero.

For $\mu \in \MM_0(\Omega)$
we consider the solution operator $L_\mu \colon H^{-1}(\Omega) \to H_0^1(\Omega)$
of the relaxed Dirichlet problem
\begin{equation*}
	y \in H_0^1(\Omega): \quad -\Delta y+\mu y=f,
\end{equation*}
that is,
$L_\mu$ maps $f \in H^{-1}(\Omega)$ to the solution $y$ of
\begin{align}
\label{weak_formulation_relaxed_BVP}
	\begin{split}
	& y \in H_0^1(\Omega) \cap L_\mu^2(\Omega):\\
	&\int_\Omega \nabla y \, \nabla z \, \d x+\int_\Omega y\, z\, \d \mu
	=\dual {f}{z}_{H^{-1}(\Omega),H_0^1(\Omega)}
	\quad \forall z \in H_0^1(\Omega) \cap L_\mu^2(\Omega).
	\end{split}
\end{align}
The solution to \eqref{weak_formulation_relaxed_BVP} exists and is unique,
it can be identified with 
the Fr\'echet-Riesz representative of 
$f \in H^{-1}(\Omega) \subset (H_0^1(\Omega)\cap L_\mu^2(\Omega))'$
with respect to the scalar product 
$(y,z)=\int_\Omega \nabla y \, \nabla z \, \d x+\int_\Omega y\, z\, \d \mu$ 
on $H_0^1(\Omega)\cap L_\mu^2(\Omega)$.
Indeed, $H_0^1(\Omega)\cap L_\mu^2(\Omega)$ is a Hilbert space, 
see \cite[Proposition~2.1]{ButtazzoDalMaso1991}.

Let us motivate the notion of ``relaxed Dirichlet problem''.
Let $O \subset \Omega$ be a quasi-open set. We define the measure
$\infty_{\Omega \setminus O}$
via
\begin{equation*}
	\infty_{\Omega \setminus O}(B)=
	\begin{cases}
		0, &\text{ if } \capa(B \setminus O)=0,\\
		+\infty, &\text{ otherwise,}
	\end{cases}
\end{equation*}
for all Borel sets $B \subset \Omega$.
By definition, $\infty_{\Omega \setminus O}$ is a Borel measure
and it is clear that $\infty_{\Omega \setminus O}$ vanishes on sets
with zero capacity.
The regularity of $\infty_{\Omega \setminus O}$ in the sense of \cref{def:capacitary_measure}
is easy to check, see \cite[Remark~3.3]{DalMaso1987}.
Hence, $\infty_{\Omega \setminus O} \in \MM_0(\Omega)$.
From the definitions, it is easy to check that
$v \in L^2_{\infty_{\Omega \setminus O}}(\Omega)$
if and only if
$v = 0$ q.e.\ on $\Omega \setminus O$
for all $v \in H_0^1(\Omega)$.
That is, $H_0^1(\Omega) \cap L^2_{\infty_{\Omega \setminus O}}(\Omega) = H_0^1(O)$.
Now, it is clear that the problem \eqref{weak_formulation_relaxed_BVP}
with $\mu = \infty_{\Omega \setminus O}$
is just a reformulation of the Dirichlet problem
$-\Delta y = f$ in $H_0^1(O)\dualspace$.
Therefore, the problems of class \eqref{weak_formulation_relaxed_BVP}
with $\mu \in \MM_0(\Omega)$
comprise the classical Dirichlet problem
on open sets, but also more general problems.

Similarly,
the problem
\begin{equation*}
	y \in H_0^1(I(u)):\quad-\Delta y +\infty_{\Omega \setminus I(u)} y=f
\end{equation*}
is an equivalent reformulation of \eqref{G\^ateaux_derivatives}.
Therefore,
the operators $L_{I(u)}$ (introduced in \eqref{G\^ateaux_derivatives}) and $L_{\infty_{\Omega \setminus I(u)}}$ (from \eqref{weak_formulation_relaxed_BVP}) coincide.
Thus, all possible Gâteaux derivatives of $S$
form a subset of $\{ L_\mu : \mu \in \MM_0(\Omega) \}$.

Next, we will describe
how the set $\MM_0(\Omega)$ can be equipped 
with a metric structure,
rendering it a metric space with nice properties.
We note that some references do not include the regularity condition
from \cref{def:capacitary_measure}
in the definition of $\MM_0(\Omega)$.
In the case that this regularity condition is dropped,
one has to consider equivalence classes of capacitary measures
in order to obtain a metric space.
For a thorough discussion of this topic, we refer to \cite[Section~3]{DalMaso1987}.

\begin{definition}
	Let $\{\mu_n\} \subset \MM_0(\Omega)$. 
	We say that the sequence $\{\mu_n\}$ $\gamma$-converges to $\mu \in \MM_0(\Omega)$ 
	if and only if 
	\begin{equation*}
		L_{\mu_n} \toWOT L_\mu \text{ in } \LL(H^{-1}(\Omega),H_0^1(\Omega)).
	\end{equation*}
	If $\{\mu_n\}$ $\gamma$-converges to $\mu$, 
	we write $\mu_n \togamma \mu$.
\end{definition}
The name $\gamma$-convergence stems from the observation
that this is closely related to the $\Gamma$-convergence of suitable functionals.
To this end,
we define $F_\mu : L^2(\Omega) \to [0,\infty]$ via
\begin{equation}
	\label{eq:functional_F_mu}
	F_\mu(u)
	:=
	\begin{cases}
		\int_\Omega \abs{\nabla u}^2 \, \dx + \int_\Omega u^2 \, \d\mu & \text{if } u \in H_0^1(\Omega) \cap L^2_\mu(\Omega) \\
		+\infty & \text{else}
	\end{cases}
\end{equation}
for all $u \in L^2(\Omega)$ and $\mu \in \MM_0(\Omega)$.
\begin{definition}
	\label{def:Gamma_convergence}
	Let $\{\mu_n\}\subset\MM_0(\Omega)$ and $\mu\in\MM_0(\Omega)$ be given.
	We say that the functionals $F_{\mu_n}$ $\Gamma$-converge towards $F_\mu$ in $L^2(\Omega)$ if and only if
	\begin{subequations}
		\label{eq:Gamma_convergence}
		\begin{align}
			\label{eq:Gamma_convergence_1}
			\forall \{u_n\} \subset L^2(\Omega) \text{ with } u_n \to u \text{ in } L^2(\Omega) :& \qquad F_\mu(u) \le \liminf_{n \to \infty} F_{\mu_n}(u_n)
			\\
			\label{eq:Gamma_convergence_2}
			\exists \{u_n\} \subset L^2(\Omega) \text{ with } u_n \to u \text{ in } L^2(\Omega) :& \qquad F_\mu(u) = \lim_{n \to \infty} F_{\mu_n}(u_n)
		\end{align}
	\end{subequations}
	hold for all $u \in L^2(\Omega)$.
	In this case, we write $F_{\mu_n} \toGamma F_\mu$ in $L^2(\Omega)$.
\end{definition}

The following lemma shows equivalent conditions for $\gamma$-convergence.
\begin{lemma}
	\label{lem:equivalencies_gamma_convergence}
	Let $\{\mu_n\} \subset \MM_0(\Omega)$ and $\mu \in \MM_0(\Omega)$ be given. 
	Then, the following statements are equivalent:\\
	\begin{minipage}[b]{.5\textwidth}
		\begin{enumerate}
			\item\label{it:equiv_1} $\mu_n \togamma \mu$.
				\stepcounter{enumi}
			\item\label{it:equiv_2} $L_{\mu_n} \toSOT L_\mu \text{ in } \LL(L^2(\Omega), L^2(\Omega))$.
				\stepcounter{enumi}
			\item\label{it:equiv_5} $L_{\mu_n}(1) \to L_\mu(1)$ in $L^2(\Omega)$.
		\end{enumerate}
	\end{minipage}%
	\begin{minipage}[b]{.5\textwidth}
		\begin{enumerate}[start=2]
			\item\label{it:equiv_6} $F_{\mu_n} \toGamma F_\mu$ in $L^2(\Omega)$.
				\stepcounter{enumi}
			\item\label{it:equiv_3} $L_{\mu_n} \toWOT L_\mu \text { in } \LL(L^2(\Omega),H_0^1(\Omega))$.
				\stepcounter{enumi}
			\item\label{it:equiv_4} $L_{\mu_n}(1) \weakly L_\mu(1)$ in $H_0^1(\Omega)$.
		\end{enumerate}
	\end{minipage}%
\end{lemma}
\begin{proof}
	Let $\mu_n \togamma \mu$. 
	Then,
	for all $f \in H^{-1}(\Omega)$,
	in particular,
	for all $f \in L^2(\Omega)$,
	it holds $L_{\mu_n}(f) \weakly L_\mu(f)$ in $H_0^1(\Omega)$.
	Since $H_0^1(\Omega)$ is compactly embedded into $L^2(\Omega)$,
	it follows $L_{\mu_n}(f) \to L_\mu(f)$ in $L^2(\Omega)$,
	and thus \ref{it:equiv_2} holds.
	
	Now, 
	suppose that $L_{\mu_n}\toSOT L_{\mu}$ in $\LL(L^2(\Omega),L^2(\Omega))$.
	Let $f \in L^2(\Omega)$.
	By \cite[(3.7)]{ButtazzoDalMaso1991},
	there is a constant $c>0$,
	such that 
	$\|L_{\mu_n}(f)\|_{H_0^1(\Omega)}\leq c \|f\|$ holds.
	Thus there is a subsequence $\{L_{\mu_{n_k}}\}$ 
	that converges weakly in $H_0^1(\Omega)$.
	Hence $\{L_{\mu_{n_k}}(f)\}$ converges strongly in $L^2(\Omega)$
	and the limit has to be $L_{\mu}(f)$.
	Thus, 
	the whole sequence $\{L_{\mu_n}(f)\}$ converges weakly to $L_{\mu}(f)$ in $H_0^1(\Omega)$
	and \ref{it:equiv_3} follows.
	The proof that \ref{it:equiv_4} follows from \ref{it:equiv_5} is also contained in this argument.
	
	\ref{it:equiv_4} is an immediate consequence of \ref{it:equiv_3}
	and \ref{it:equiv_5} follows from \ref{it:equiv_4} by the compact embedding of $H_0^1(\Omega)$ into $L^2(\Omega)$.
	
	The equivalence of \ref{it:equiv_4} and \ref{it:equiv_1} has been shown, in a more general setting, in \cite[Theorem~5.1]{DalMasoMurat2004}.
	
	The equivalence between \ref{it:equiv_2} and \ref{it:equiv_6}
	can be checked as in \cite[Proposition~4.10]{DalMasoMosco1987}.
\end{proof}
Using the equivalence of $\mu_n \togamma \mu$
and $\norm{ L_{\mu_n}(1) - L_\mu(1)}_{L^2(\Omega)} \to 0$,
we can equip $\MM_0(\Omega)$ with a metric.
\begin{corollary}
	The $\gamma$-convergence on $\MM_0(\Omega)$ is metrizable.
\end{corollary}
A different proof of this metrizability can be found in \cite[Proposition~4.9]{DalMasoMosco1987}.

The metric space $\MM_0(\Omega)$ has many nice properties:
it is complete (\cref{lem:completeness}),
the subset $\{ \infty_{\Omega \setminus O} : O \subset \Omega \text{ is quasi-open}\}$ is dense (\cref{lem:quasi_open_sets_dense})
and $\MM_0(\Omega)$ is compact (\cref{thm:compactness}).
\begin{lemma}
\label{lem:completeness}
	The metric space $\MM_0(\Omega)$ is complete.
\end{lemma}
For a proof, we refer to
\cite[Theorem~4.14]{DalMasoMosco1987}
or
\cite[Theorem~4.5]{DalMasoGarroni1997}.

The next lemma shows
that the measures $\infty_{C}$ with a quasi-closed set $C \subset \Omega$
represent a dense subclass of $\MM_0(\Omega)$.

\begin{lemma}
\label{lem:quasi_open_sets_dense}
	Let $\mu$ be an element of $\MM_0(\Omega)$. 
	Then there is a sequence $\{O_n\}_{n \in \mathbb{N}}\subset \Omega$ 
	of quasi-open sets
	such that $\infty_{\Omega \setminus O_n} \togamma \mu$.
\end{lemma}

A proof can be found in \cite[Theorem~4.16]{DalMasoMosco1987}
and a more constructive argument is given in
\cite{DalMasoMalusa1995}.

The preceding lemma shows the connection 
between capacitary measures and shape optimization problem.
Due to the fact that solutions of classical Dirichlet problems
with varying (quasi-open) domains
can converge to the solution of a relaxed Dirichlet problem
with capacitary measures involved,
an optimal domain in shape optimization might not exist,
see e.g.\ \cite[Section~4.2]{BucurButtazzo2005} or \cite[Section~5.8.4]{AttouchButtazzoMichaille2014}.

The next theorem shows the compactness of $\MM_0(\Omega)$.
\begin{theorem}
\label{thm:compactness}
	Let $\{\mu_n\}$ be a sequence in $\MM_0(\Omega)$. 
	Then there exists a subsequence $\{\mu_{n_k}\}$ 
	and a measure $\mu \in \MM_0(\Omega)$ 
	such that $\mu_{n_k} \togamma \mu$.
\end{theorem}

For a proof, we refer to 
\cite[Theorem~4.14]{DalMasoMosco1987}.
Therein, one has to replace $\mathbb{R}^n$ by $\Omega$ to obtain the desired result.

Many properties of capacitary measures
can be obtained by studying the so-called torsion function $w_\mu := L_\mu(1)$.
Indeed, we have already seen in \cref{lem:equivalencies_gamma_convergence}
that it is sufficient to check the convergence $w_{\mu_n} \to w_\mu$ in $L^2(\Omega)$
of the torsion functions
to obtain $\mu_n \togamma \mu$.
This implies in particular,
that
the measure $\mu$ is uniquely determined
by its torsion function,
see also
\cite[Proposition~3.4]{DalMasoGarroni1994}
and \cite[Theorem~1.20]{DalMasoGarroni1997}.

Moreover,
the next result shows that
the torsion function $w$ associated with a quasi-open set $O \subset \Omega$
is positive on $O$, whereas the fine support of $1 + \Delta w$ is $\Omega \setminus O$.
\begin{theorem}
	\label{thm:torsion_function_properties}
	Let $O \subset \Omega$ be quasi-open and set $w := L_O(1)$.
	Then,
	$w \geq 0$,
	$O =_q \{w > 0\}$
	and
	$1 + \Delta w \in H^{-1}(\Omega)^+$
	with
	$\fsupp(1 + \Delta w) =_q \Omega \setminus O$.
\end{theorem}
\begin{proof}
	It holds $w \geq 0$ by \cite[Proposition~2.4]{DalMasoGarroni1994}.
	The assertions $O =_q \{w > 0\}$
	and $1 + \Delta w \in H^{-1}(\Omega)^+$ are well known,
	see, e.g., \cite[Proposition~3.4.26]{Velichkov2015} and \cite[Theorem~1]{ChipotDalMaso1992}.

	It remains to check
	$C :=_q \fsupp(1 + \Delta w) =_q \Omega \setminus O$.
	Using the characterization of 
	\cref{lem:fine_support},
	%\cite[Lemma~3.7]{HarderWachsmuth2017:2},
	we have
	\begin{equation}
		\label{eq:characterization_fine_support_2}
		\dual{1 + \Delta w}{v} = 0
		\quad\Leftrightarrow\quad
		v = 0 \text{ q.e.\ on } C
		\qquad
		\forall v \in H_0^1(\Omega)^+.
	\end{equation}
	Using $w = L_O(1)$, this directly implies that
	$C \subset_q \Omega \setminus O$.
	Next, we define
	$\hat w = L_{\Omega \setminus C}(1)$.
	Since $w \in H_0^1(O) \subset H_0^1(\Omega \setminus C)$,
	we have
	$\dual{1 + \Delta \hat w}{w} = 0$.
	Moreover, \eqref{eq:characterization_fine_support_2}
	implies
	$\dual{1 + \Delta w}{\hat w} = 0$.
	Using $\dual{\Delta w}{\hat w} = \dual{\Delta \hat w}{w}$,
	this implies
	\begin{equation*}
		\dual{1}{\hat w - w}
		=
		\int_\Omega \hat w - w \, \dx
		=
		0.
	\end{equation*}
	Next, the comparison principle
	from \cite[Theorem~2.10]{DalMasoMosco1986},
%	\alert{AT: Quelle oben gibt es noch nicht oder? Hab eine neu angelegt, die meintest du doch, oder? GW: Ja, danke!}
	see also \cite[Proposition~2.5]{DalMasoGarroni1994},
	implies
	$\hat w \ge w$
	and, therefore,
	$\hat w = w$.
	Finally, the first part of the proof
	yields
	$\Omega \setminus C =_q \{\hat w > 0\} =_q \{w > 0\} =_q O$.
	Thus,
	$C =_q \fsupp(1+\Delta w) =_q \Omega \setminus O$.
\end{proof}

The next result shows that every capacitary measure
can be approximated by Radon measures.
Here,
a Radon measure is a Borel measure which is finite on all compact subsets of $\Omega$.
\begin{lemma}
	\label{lem:approximation_Radon}
	Let $\mu \in \MM_0(\Omega)$. 
	Then there exists an increasing sequence of Radon measures $\{\mu_n\}$ 
	such that $\mu_n \togamma \mu$.
\end{lemma}

\begin{proof}
	Let $w_0:=L_\Omega(1)$
	and for $\mu \in \MM_0(\Omega)$ 
	let $w:=L_\mu(1)$.
	In \cite[Proposition~4.7]{DalMasoGarroni1994}, 
	it is shown that for the sequence $\{w_n\}$ defined by
	\begin{equation*}
	w_n:=\left(1-\frac{1}{n}\right)w+\frac{1}{n} w_0
	\end{equation*}
	the associated measures defined by
	\begin{equation*}
	\mu_n(B):= 
	\begin{cases} 
	\int_B \frac{\d(1+\Delta w_n)}{w_n}, 
	& \text{ if } \capa(B\cap \{w_n=0\})=0,\\ 
	+\infty, & \text{ else, } 
	\end{cases}
	\end{equation*}
	are Radon measures $\gamma$-converging to $\mu$.
	Thus, it remains to show the monotonicity of this sequence.
	Since $w_0>0$ by \cref{thm:torsion_function_properties},
	it holds $\mu_n(B) = \int_B \frac{\d(1+\Delta w_n)}{w_n}$ for all $n \in \mathbb{N}$ and for all Borel sets $B$.
	The representation
	\begin{align*}
	\mu_n(B)&=\int_B \frac{\d (1+\Delta w_n)}{w_n}=\int_{B} \frac{\d (1+\Delta ((1-1/n)w+1/n\,w_0))}{(1-1/n)w+1/n\,w_0}\\
	&=\int_B \frac{\d (1-1/n+(1-1/n)\Delta w)}{(1-1/n)(w+1/(n-1)\,w_0)}
	=\int_B \frac{\d(1+\Delta w)}{w+1/(n-1)\,w_0}
	\end{align*}
	shows that
	$\mu_n\leq \mu_{n+1}\leq \mu$ holds for all $n \in \mathbb{N}$.
\end{proof}

The following lemma shows 
that the image of $L_\mu$ is dense in $H_0^1(\Omega) \cap L_\mu^2(\Omega)$.

\begin{lemma}
	\label{lem:density_image_L_mu}
	Let $\mu \in \MM_0(\Omega)$ 
	and let $y \in H_0^1(\Omega) \cap L_\mu^2(\Omega)$. 
	Then there is a sequence 
	\begin{equation*}
		\{y_n\} \subset \{L_\mu(f): f\in H^{-1}(\Omega)\}
	\end{equation*}
	such that $y_n \to y$ in $H_0^1(\Omega) \cap L_\mu^2(\Omega)$.
\end{lemma}

\begin{proof}
	For every $n \in \mathbb{N}$
	let $y_n \in H_0^1(\Omega) \cap L_\mu^2(\Omega)$
	be the solution of the problem
	\begin{equation*}
		\int_{\Omega}\nabla y_n\,\nabla v\,\d x
		+\int_{\Omega}y_n\,v\,\d\mu=
		-n\int_{\Omega} (y_n-y)\,v\,\d x 
		\quad \forall v \in H_0^1(\Omega) \cap L_\mu^2(\Omega).
	\end{equation*}
	We can write $y_n=L_\mu(-n(y_n-y))$, 
	thus $y_n \in \{L_\mu(f): f\in H^{-1}(\Omega)\}$.
	By \cite[Proposition~3.1]{DalMasoGarroni1994},
	it holds $y_n \to y$ in $H_0^1(\Omega) \cap L_\mu^2(\Omega)$
	and the conclusion follows.
\end{proof}

\begin{lemma}
	\label{lem:v=0_on_w=0}
	Let $\mu \in \MM_0(\Omega)$ 
	and assume that $v \in H_0^1(\Omega)\cap L_\mu^2(\Omega)$. 
	Then it holds 
	$v=0$ q.e.\ on $\{w_\mu=0\}$ 
	and $v \in H_0^1(\{w_\mu>0\})$.
\end{lemma}

\begin{proof}
	By \cite[Proposition~3.4]{DalMasoGarroni1994}, 
	it holds $\mu(B)=+\infty$ for all Borel sets $B \subset \Omega$ with $\capa(B \cap \{w_\mu=0\})>0$. 
	Thus $v=0$ q.e.\ on $\{w_\mu=0\}$ for all $v$ in the image of $L_\mu$.
	By density of this set in $H_0^1(\Omega) \cap L_\mu^2(\Omega)$,
	see \cref{lem:density_image_L_mu},
	it follows $v=0$ q.e.\ on $\{w_\mu=0\}$ for all $v \in H_0^1(\Omega) \cap L_\mu^2(\Omega)$.
	
	It holds $w_\mu \geq 0$ on $\Omega$ by \cref{thm:torsion_function_properties},
	therefore,
	each $v \in H_0^1(\Omega) \cap L_\mu^2(\Omega)$ is in $H_0^1(\{w_\mu>0\})$.
\end{proof}

The next result characterizes the completion
of $H_0^1(\Omega) \cap L^2_\mu(\Omega)$
in
$H_0^1(\Omega)$.
\begin{lemma}
	\label{lem:completion_h01_l2}
	Let $\mu \in \MM_0(\Omega)$ be given.
	Then,
	\begin{equation*}
		\overline{ H_0^1(\Omega) \cap L^2_\mu(\Omega) }^{H_0^1(\Omega)}
		=
		H_0^1( \{w_\mu > 0\} ).
	\end{equation*}
	Moreover, for any $v \in H_0^1( \{w_\mu > 0\} )^+$,
	there exists a sequence $\{v_n\}_{n \in \N} \subset H_0^1(\Omega) \cap L^2_\mu(\Omega)$
	such that
	$0 \le v_n \le v$ q.e.\ on $\Omega$ for all $n \in \N$
	and
	$v_n \to v$ in $H_0^1(\Omega)$.
\end{lemma}
\begin{proof}
	We set
	$V := \overline{ H_0^1(\Omega) \cap L^2_\mu(\Omega) }^{H_0^1(\Omega)}$.
	The inclusion $V \subset H_0^1( \{w_\mu > 0\} )$
	is clear from \cref{lem:v=0_on_w=0}.

	Then, it can be checked that $V$
	is a closed lattice ideal in $H_0^1(\Omega)$,
	i.e., it is a closed subspace with the property
	that
	$v \in V$, $w \in H_0^1(\Omega)$ and $\abs{w} \le \abs{v}$
	imply $w \in V$.
	Hence,
	\cite{Stollmann1993} implies
	that $V = H_0^1(\tilde \Omega)$ for some quasi-open $\tilde\Omega \subset \Omega$.
	Thus,
	\begin{equation*}
		w_\mu
		\in
		V
		=
		H_0^1(\tilde \Omega)
		\subset
		H_0^1( \{w_\mu > 0\})
	\end{equation*}
	and
	together with \cref{thm:torsion_function_properties}
	we get
	$\tilde\Omega =_q \{w_\mu > 0\}$.
	This shows
	$\overline{ H_0^1(\Omega) \cap L^2_\mu(\Omega) }^{H_0^1(\Omega)} = V = H_0^1(\tilde\Omega) = H_0^1(\{w_\mu > 0\})$.

	The second assertion is clear since $w \mapsto \max\bigh(){0, \min(w, v)}$ is continuous on $H_0^1(\Omega)$.
\end{proof}
Note that a similar assertion
which, however,
uses the so-called singular set of the measure $\mu$
can be found in
\cite[Lemma~2.6]{ButtazzoDalMaso1991}.

The next lemma shows that the solution operators associated with quasi-open sets
form a (sequentially) closed set w.r.t.\ SOT.
\begin{lemma}
	\label{lem:SOT_convergence}
	Let $\Omega_n \subset \Omega$ be a sequence of quasi-open sets such that
	$L_{\Omega_n}$ converges in the SOT towards some $L \in \LL(H^{-1}(\Omega), H_0^1(\Omega))$.
	Then, the limit satisfies $L = L_{\hat\Omega}$
	for some quasi-open set $\hat\Omega \subset \Omega$.
\end{lemma}
\begin{proof}
	From \cref{lem:completeness}
	we know that $L = L_\mu$ for some $\mu \in \MM_0(\Omega)$.
	For given $f \in H^{-1}(\Omega)$,
	we set $v_n := L_{\Omega_n} f$.
	Then, $v_n \to v := L_\mu f$ in $H_0^1(\Omega)$
	and this yields
	\begin{equation*}
		\int_\Omega v^2 \, \d\mu
		=
		\int_\Omega -\abs{\nabla v}^2 + f \, v \, \d x
		=
		\lim_{n \to \infty}
		\int_\Omega -\abs{\nabla v_n}^2 + f \, v_n \, \d x
		=
		0.
	\end{equation*}
	Hence, $\int_\Omega v^2 \, \d\mu = 0$ for all $v$ in the range of $L_\mu$.

	In order to check $L_\mu = L_{\hat\Omega}$
	for some quasi-open set $\hat\Omega \subset \Omega$,
	we use the torsion function $w = L_\mu(1)$
	and set
	$\hat\Omega :=_q \{w > 0\}$.
	From $\int_\Omega w^2 \, \d\mu = 0$
	and $v \in H_0^1(\hat\Omega)$ for $v \in H_0^1(\Omega) \cap L^2_\mu(\Omega)$ (see \cref{lem:v=0_on_w=0}), it follows that
	$w = L_{\hat\Omega}(1)$.
	Thus, $L_{\hat\Omega} = L_\mu$ by
	\cite[Proposition~3.4]{DalMasoGarroni1994}.
\end{proof}
Note that we even have the following converse of \cref{lem:SOT_convergence}.
If $L_{\Omega_n} \toWOT L_\Omega$,
then we already get $L_{\Omega_n} \toSOT L_\Omega$,
see \cite[Proposition~5.8.6]{AttouchButtazzoMichaille2014}.
That is,
the $\gamma$-limit of the sequence of quasi-open sets $\{\Omega_n\}$
is again a quasi-open set
if and only if the solution operators converge in the strong operator topology.

Let us also mention that the $\gamma$-convergence 
of a sequence of quasi-open sets $\{\Omega_n\}$ to a quasi-open set $\tilde{\Omega}$,
i.e., the convergence $L_{\Omega_n} \toSOT L_{\tilde{\Omega}}$,
is equivalent to the convergence of the spaces $\{H_0^1(\Omega_n)\}$ to $H_0^1(\tilde{\Omega})$
in the sense of Mosco,
see \cite[Prop.~4.53, Remark~4.5.4]{BucurButtazzo2005}.
This tool is also used in the derivation of a generalized gradient in \cite{RaulsUlbrich2018}.

As a last result in this section,
we are going to study the convergence of
a sum of two $\gamma$-convergent sequences.
To this end, we need an auxiliary lemma.
\begin{lemma}
	\label{lem:h01_weak_convergence_and_norms}
	Let $\{u_n\}, \{v_n\} \subset H_0^1(\Omega)$
	be sequences with $u_n \to u$ in $H_0^1(\Omega)$ and $v_n \weakly u$ in $H_0^1(\Omega)$
	for some $u \in H_0^1(\Omega)$.
	% We further assume that $\lim_{n \to \infty} \norm{v_n}_{H_0^1(\Omega)}$ exists.
	Then, $w_n := \min(u_n, v_n)$ satisfies
	$w_n \weakly u$ in $H_0^1(\Omega)$
	and
	\begin{equation*}
		\limsup_{n \to \infty}
		\bigh(){
		\norm{w_n}_{H_0^1(\Omega)}^2
		-
		\norm{v_n}_{H_0^1(\Omega)}^2
		}
		\le
		0
		.
	\end{equation*}
\end{lemma}
\begin{proof}
	The weak convergence of $w_n$ follows from the weak sequential continuity of $\min(\cdot, \cdot)$
	in $H_0^1(\Omega)$.
	To obtain the desired inequality,
	we check
	\begin{align*}
		\norm{w_n}_{H_0^1(\Omega)}^2 - \norm{v_n}_{H_0^1(\Omega)}^2
		&=
		\norm{w_n - u_n}_{H_0^1(\Omega)}^2 - \norm{v_n - u_n}_{H_0^1(\Omega)}^2  + 2 \, \scalarprod{w_n - v_n}{u_n}_{H_0^1(\Omega)}
		\\
		% &=
		% \norm{\min(0, v_n - u_n)}_{H_0^1(\Omega)}^2 - \norm{v_n - u_n}_{H_0^1(\Omega)}^2  + 2 \, \scalarprod{w_n - v_n}{u_n}_{H_0^1(\Omega)}
		% \\
		&=
		-\norm{\max(0, v_n - u_n)}_{H_0^1(\Omega)}^2 + 2 \, \scalarprod{w_n - v_n}{u_n}_{H_0^1(\Omega)}
		\\
		&\le
		2 \, \scalarprod{w_n - v_n}{u_n}_{H_0^1(\Omega)}.
	\end{align*}
	Now, the claim follows from $w_n - v_n \weakly 0$ and $u_n \to u$ in $H_0^1(\Omega)$.
\end{proof}

\begin{theorem}
	\label{thm:gamma_limit_of_sum_of_measures}
	Let $\{\mu_n\}$ be a sequence in $\MM_0(\Omega)$
	such that
	$\mu_n \togamma \mu$
	and
	let $\{C_n\}$ be a sequence of quasi-closed subsets of $\Omega$
	such that
	$\infty_{C_n} \togamma \infty_C$
	for some quasi-closed set $C \subset \Omega$.
	Then,
	$\mu_n + \infty_{C_n} \togamma \mu + \infty_C$.
\end{theorem}
\begin{proof}
	We use the characterization of $\gamma$-convergence via
	the $\Gamma$-convergence of the functionals $F_{\mu_n + \infty_{C_n}}$.
	Therefore, we have to verify \eqref{eq:Gamma_convergence}.
	Let $u \in L^2(\Omega)$ be given
	and consider an arbitrary sequence $\{u_n\} \subset L^2(\Omega)$
	with $u_n \to u$ in $L^2(\Omega)$.
	We have to show
	\begin{equation*}
		F_{\mu + \infty_C}(u)
		\le
		\liminf_{n \to \infty}
		F_{\mu_n + \infty_{C_n}}(u_n).
	\end{equation*}
	If the limes inferior is $+\infty$, there is nothing to show.
	Otherwise, we select a subsequence of $\{u_n\}$
	(without relabeling),
	such that the limes inferior is actually a limit
	and such that $F_{\mu_n + \infty_{C_n}}(u_n) < +\infty$ for all $n$.
	This implies
	$u_n \in H_0^1(\Omega)$
	as well as
	$\int_\Omega u_n^2 \, \d\infty_{C_n} < +\infty$,
	and these properties
	yield
	$u_n \in H_0^1(\Omega \setminus C_n)$.
	Consequently, we have
	\begin{equation*}
		F_{\infty_C}(u)
		\le
		\liminf_{n \to \infty}
		F_{\infty_{C_n}}(u_n)
		\le
		\liminf_{n \to \infty}
		F_{\mu_n + \infty_{C_n}}(u_n)
		< +\infty.
	\end{equation*}
	Thus, $u \in H_0^1(\Omega \setminus C)$
	and $\int_\Omega u^2 \, \d\infty_C = 0$.
	Now, the desired inequality follows by
	\begin{equation*}
		F_{\mu + \infty_C}(u)
		=
		F_{\mu}(u)
		\le
		\liminf_{n \to \infty}
		F_{\mu_n}(u_n)
		=
		\liminf_{n \to \infty}
		F_{\mu_n + \infty_{C_n}}(u_n),
	\end{equation*}
	where we have used $F_{\mu_n} \toGamma F_\mu$.

	Further, we have to prove the existence of a sequence $\{w_n\} \subset L^2(\Omega)$
	with $w_n \to u$ in $L^2(\Omega)$
	and
	\begin{equation*}
		F_{\mu + \infty_C}(u)
		=
		\lim_{n \to \infty}
		F_{\mu_n + \infty_{C_n}}(w_n).
	\end{equation*}
	It is enough to consider the case $u \ge 0$,
	otherwise apply the following arguments to $u^+$ and $u^-$.
	If $F_{\mu + \infty_C}(u)=\infty$, there is nothing to show.
	Otherwise, we have $u \in H_0^1(\Omega \setminus C)$.
	From $F_{\mu_n} \toGamma F_\mu$ and $F_{\infty_{C_n}} \toGamma F_{\infty_C}$,
	we find sequences $\{v_n\}, \{u_n\} \subset L^2(\Omega)$
	with
	\begin{align*}
		v_n \to u \text{ in } L^2(\Omega) &\quad\text{and}\quad F_{\mu_n}(v_n) \to F_\mu(u), \\
		u_n \to u \text{ in } L^2(\Omega) &\quad\text{and}\quad F_{\infty_{C_n}}(u_n) \to F_{\infty_C}(u).
	\end{align*}
	W.l.o.g., we can assume $v_n, u_n \ge 0$
	(otherwise, replace $v_n$ by $\max(v_n,0)$ and $u_n$ by $\max(u_n,0)$).
	We easily infer $v_n \weakly u$ in $H_0^1(\Omega)$ and $u_n \to u$ in $H_0^1(\Omega)$.
	We define $w_n = \min(u_n, v_n)$ and already get $w_n \to u$ in $L^2(\Omega)$.
	To obtain the convergence of the function values,
	we use $w_n = 0$ q.e.\ on $C_n$
	to obtain
	\begin{align*}
		F_{\mu_n + \infty_{C_n}}(w_n)
		&=
		F_{\mu_n}(w_n)
		= \int_\Omega \abs{\nabla w_n}^2 \, \dx + \int_\Omega w_n^2 \, \d\mu_n
		\\
		&\le \int_\Omega \abs{\nabla w_n}^2 \, \dx + \int_\Omega v_n^2 \, \d\mu_n
		= F_{\mu_n}(v_n) +
		\bigh(){
		\norm{w_n}_{H_0^1(\Omega)}^2
		-
		\norm{v_n}_{H_0^1(\Omega)}^2
		}
		.
	\end{align*}
	Now, by using \cref{lem:h01_weak_convergence_and_norms} and $u \in H_0^1(\Omega \setminus C)$
	we obtain
	\begin{align*}
		F_{\mu + \infty_{C}}(u)
		&\le
		\liminf_{n \to \infty}
		F_{\mu_n + \infty_{C_n}}(w_n)
		\\
		&\le
		\limsup_{n \to \infty}
		F_{\mu_n + \infty_{C_n}}(w_n)
		\le
		\limsup_{n \to \infty}
		F_{\mu_n}(v_n)
		=
		F_\mu(u)
		=
		F_{\mu + \infty_C}(u).
	\end{align*}
	Thus,
	$F_{\mu + \infty_C}(u) = \lim_{n \to \infty} F_{\mu_n + \infty_{C_n}}(w_n)$.
	This finishes the proof of
	$F_{\mu_n + \infty_{C_n}}(w_n) \toGamma F_{\mu + \infty_C}(u)$.
\end{proof}
% \alert{
% 	This might be false:
% 	Note that there might be a possibility to strengthen the above result
% 	as follows.
% 	Assume that $\{\mu_n\}, \{\nu_n\}$ are sequences
% 	with
% 	$\mu_n \togamma \mu$ and $\nu_n \togamma \nu$.
% 	Further, we suppose that for all $u \in H_0^1(\Omega)$
% 	there exists a sequence $u_n$ with $u_n \to u$ in $H_0^1(\Omega)$
% 	and $F_\nu(u) = \lim_{n \to \infty} F_{\nu_n}(u_n)$.
% 	Note that this is just $F_{\nu_n} \toGamma F_\nu$ in $H_0^1(\Omega)$.
% 	Then,
% 	$F_{\frac12 \, (\mu_n + \nu_n)} \toGamma F_{\frac12 \, (\mu+\nu)}$ in $L^2(\Omega)$.
% 
% The above would be true, if \cref{lem:h01_weak_convergence_and_norms} could be strengthened to
% \begin{equation*}
% 	\limsup_{n \to \infty}
% 	\bigh(){
% 		\norm{w_n}_{H_0^1(\Omega)}^2
% 		-
% 		\frac12 \, \norm{v_n}_{H_0^1(\Omega)}^2
% 		-
% 		\frac12 \, \norm{u_n}_{H_0^1(\Omega)}^2
% 	}
% 	\le
% 	0
% 	,
% \end{equation*}
% but this does not seem to be the case
% }
%By choosing $C_n = C = K$,
%this theorem generalizes \cref{lem:convergence_on_open_subsets}
%to quasi-closed sets $K \subset \Omega$.

\section{Generalized derivatives involving the SOT}
\label{sec:generalized_derivatives_SOT}
In this section, we are going to characterize
the generalized derivatives of the obstacle problem
which involve the SOT.

Therefore,
as a technique,
we frequently use the argument that
if $\hat\Omega \subset \Omega$ is quasi-open
and if $v \in H_0^1(\hat{\Omega})$,
then this implies $v=L_{\hat{\Omega}}(-\Delta v)$.

As a first result, we give an upper estimate for
$\pBws S(u)$.

\begin{lemma}
	\label{lem:upper_estimate_pBws}
	Let us assume that $L \in \pBws S(u)$.
	Then, there exists a quasi-open set $\hat\Omega \subset \Omega$ with
	$\capa(\hat\Omega \cap A_s(u)) = 0$
	and $L = L_{\hat\Omega}$.
\end{lemma}
\begin{proof}
	By definition, there is a sequence $\{u_n\} \subset D_S$
	such that $u_n \weakly u$ in $H^{-1}(\Omega)$,
	$S(u_n) \weakly S(u)$ in $H_0^1(\Omega)$
	and $S'(u_n) \toSOT L$.
	By the characterization of differentiability points of $S$,
	we have $S'(u_n) = L_{I(u_n)}$.
	From \cref{lem:SOT_convergence}, we already
	know that $L = L_{\hat\Omega}$ for some quasi-open set $\hat\Omega \subset \Omega$.
	It remains to check
	$\capa(\hat\Omega \cap A_s(u)) = 0$.

	From \cref{thm:torsion_function_properties} we infer the existence of $v \in H_0^1(\Omega)^+$
	with $\{v > 0\} =_q \hat\Omega$.
	In particular, $v \in H_0^1(\hat\Omega)$
	and this yields
	that $v = L_{\hat\Omega}(-\Delta v)$
	is the strong limit of 
	$v_n := S'(u_n)(-\Delta v)$.
	By the properties of $S'(u_n)$, we have
	$v_n = 0$ q.e.\ on $A_s(u_n)$.
	Thus,
	\begin{equation*}
		\dual{\xi_n}{\abs{v_n}} = 0,
	\end{equation*}
	where $\xi_n := -\Delta S(u_n) - u_n$.
	From $\xi_n \weakly \xi := -\Delta S(u) - u$ we infer
	\begin{equation*}
		\dual{\xi}{\abs{v}} = 0.
	\end{equation*}
	Thus, $v = 0$ q.e.\ on $A_s(u)$.
	Hence,
	$\capa( \hat\Omega \cap A_s(u) ) = \capa( \{v \ne 0\} \cap A_s(u) ) = 0$.
\end{proof}

Before we can give a precise characterization
of $\pBws S(u)$ and $\pBss S(u)$,
we need an auxiliary lemma.

\begin{lemma}
	\label{lem:approximation_on_function_on_inactive_set}
	Let a sequence $u_n \to u$ in $H^{-1}(\Omega)$ be given.
	Then, for every $v \in H_0^1(I(u))$ with $0 \le v \le 1$, there exists a sequence $\{v_n\}$
	with $v_n \in H_0^1(I(u_n))$ and $v_n \to v$ in $H_0^1(\Omega)$.
\end{lemma}
\begin{proof}
	We set $y = S(u)$ and $y_n = S(u_n)$.
	Let $t_n := \sup_{m = n,\ldots,\infty} \norm{y_m - y}_{H_0^1(\Omega)}^{1/2}$.
	Then, $\{t_n\}$ is a
	decreasing sequence of nonnegative numbers
	with $t_n \ge \norm{y_n - y}_{H_0^1(\Omega)}^{1/2}$
	and $t_n \searrow 0$.
	We have
	$
		\{ y > \psi \}
		=_q
		\bigcup_{n = 1}^\infty \{ y > \psi + t_n \}.
	$
	Since the sets on the right-hand side are quasi-open and increasing in $n$,
	we can apply
	\cref{lem:quasi-covering}.
	This yields a sequence $\{\tilde v_n\} \subset H_0^1(\Omega)$
	with
	$\tilde v_n \to v$ in $H_0^1(\Omega)$,
	$0 \le \tilde v_n \le 1$
	and $\tilde v_n = 0$ q.e.\ on $\{y \le \psi + t_n\}$.

	Next, we have
	\begin{align*}
		\capa\bigh(){ \{y_n = \psi\} \cap \{y > \psi + t_n \} }
		&\le
		\capa(\{\abs{y_n - y} > t_n\})
		\le
		t_n^{-2} \, \norm{y_n - y}_{H_0^1(\Omega)}^2
		\to
		0.
	\end{align*}
	Thus, there exists $w_n \in H_0^1(\Omega)$
	with
	$w_n \to 0$ in $H_0^1(\Omega)$,
	$0 \le w_n \le 1$
	and
	$w_n = 1$ q.e.\ on $\{\abs{y_n - y} > t_n\}$.
	We set $v_n := \max(\tilde v_n - w_n, 0)$.
	By construction, $v_n \to v$
	and $v_n = 0$ q.e.\ on $\{y_n = \psi\}$,
	i.e., $v_n \in H_0^1(I(u_n))$.
\end{proof}

Next, we give a characterization of $\pBss S(u)$.
\begin{theorem}
	\label{thm:characterization_pBss}
	Let $u \in H^{-1}(\Omega)$ be given.
	Then,
	\begin{equation*}
		\pBss S(u)
		=
		\{
			L_{\hat\Omega}
			\mid
			\hat\Omega \text{ is quasi-open and }
			I(u) \subset_q \hat\Omega \subset_q \Omega \setminus A_s(u)
		\}.
	\end{equation*}
\end{theorem}
\begin{proof}
	``$\subset$'':
	Let $L \in \pBss S(u)$ be given.
	By definition, $S'(u_n) \toSOT L$ for some sequence $\{u_n\} \subset D_S$
	with $u_n \to u$.
	From \cref{lem:upper_estimate_pBws} and $L \in \pBss S(u) \subset \pBws S(u)$, we already have
	$L = L_{\hat\Omega}$ for some quasi-open $\hat\Omega \subset_q \Omega \setminus A_s(u)$.
	It remains to check $I(u) \subset_q \hat\Omega$.

	By \cref{thm:torsion_function_properties},
	there is a function $v \in H_0^1(\Omega)$ with $0 \le v \le 1$ and $I(u) =_q \{v > 0\}$.
	From \cref{lem:approximation_on_function_on_inactive_set},
	we get a sequence $\{v_n\}$ with $v_n \to v$
	and $v_n \in H_0^1(I(u_n))$.
	Together with \cref{lem:SOT_und_WOT}, we find
	\begin{equation*}
		v =
		\lim_{n \to \infty} v_n
		=
		\lim_{n \to \infty} S'(u_n) (-\Delta v_n)
		=
		L_{\hat\Omega} (-\Delta v).
	\end{equation*}
	This gives $I(u) =_q \{v > 0\} \subset_q \hat\Omega$.

	``$\supset$'':
	Let $\hat\Omega$ be given as in the formulation of the theorem.
	From \cref{thm:torsion_function_properties},
	we get a function $v \in H_0^1(\Omega)^+$ with $\{v > 0\}=_q\hat\Omega$.
	Similarly,
	\cref{thm:torsion_function_properties}
	gives $\lambda \in H^{-1}(\Omega)^+$
	with
	$\fsupp(\lambda) =_q \Omega \setminus \hat\Omega$.
	We define
	$u_n := u - (\Delta v + \lambda) / n$.
	Let us check that $y_n := y + v / n$
	satisfies $y_n = S(u_n)$.
	From $v \ge 0$, we infer $y_n \in K$.
	Further, for arbitrary $z \in K$ we have
	\begin{align*}
		\dual{ - \Delta y_n - u_n}{z - y_n}
		&=
		\Bigdual{
			-\Delta y - \frac1n \Delta v - u + \frac1n \Delta v + \frac1n \lambda
		}{ z - y - \frac1n v }
		\\
		&=
		\dual{ -\Delta y - u}{ z - y}
		+
		\Bigdual{ -\Delta y - u}{- \frac1n v }
		+
		\Bigdual{\frac1n \lambda }{ z - y - \frac1n v }
		\\
		&\ge 0 + 0 + 0.
	\end{align*}
	The second term is zero due to $\xi = -\Delta y - u$, $\fsupp(\xi) =_q A_s(u)$
	and $v = 0$ on $\Omega \setminus \hat\Omega \supset_q A_s(u)$.
	Similarly, the third term is non-negative
	since $\fsupp(\lambda) =_q \Omega \setminus \hat\Omega$
	and
	$z \ge \psi = y + v/n$ on $\Omega\setminus\hat\Omega$.
	Hence, $y_n = S(u_n)$
	and $\xi_n := -\Delta y_n - u_n = \xi + \lambda / n$.
	Thus,
	$I(u_n) =_q \hat\Omega =_q \Omega \setminus A_s(u_n)$,
	i.e., $u_n \in D_S$.
	Finally,
	$S'(u_n) = L_{\hat\Omega}$ and $u_n \to u$
	ensure $L_{\hat\Omega} \in \pBss S(u)$.
\end{proof}

We can also give a characterization of $\pBws S(u)$.
\begin{theorem}
	\label{thm:characterization_pBws}
	Let $u \in H^{-1}(\Omega)$ be given.
	Then,
	\begin{equation*}
		\pBws S(u)
		=
		\{
			L_{\hat\Omega}
			\mid
			\hat\Omega \text{ is quasi-open and }
			I(u) \subset_q \hat\Omega \subset_q \Omega \setminus A_s(u)
		\}.
	\end{equation*}
\end{theorem}
\begin{proof}
	``$\subset$'':
	Let $L \in \pBws S(u)$ be given.
	By definition, $S'(u_n) \toSOT L$ for some sequence $\{u_n\} \subset D_S$
	with $u_n \weakly u$ and $S(u_n) \weakly S(u)$.
	From \cref{lem:upper_estimate_pBws}, we already have
	$L = L_{\hat\Omega}$ for some quasi-open $\hat\Omega \subset_q \Omega \setminus A_s(u)$.
	It remains to check $I(u) \subset_q \hat\Omega$.

	We set $w = L_{\hat\Omega} 1$
	and $w_n = S'(u_n) 1=L_{I(u_n)} 1$.
	From \cref{thm:torsion_function_properties},
	we find
	$1 + \Delta w_n \ge 0$ and $\fsupp(1 + \Delta w_n) =_q A(u_n)$.
	Since $y_n := S(u_n) = \psi$ q.e.\ on $A(u_n)$
	and since $y_n$ and $\psi$ are assumed to be Borel measurable, this gives
	\begin{equation*}
		\int_\Omega (y_n - \psi) \, \d( 1+ \Delta w_n) = 0.
	\end{equation*}

	In the next few lines, we need to work with a capacity on all of $\R^d$.
	This can be defined as in \cite[Section~1]{DalMaso1983}.
	The function $y - \psi$ is non-negative and quasi lower-semicontinuous.
	Moreover, if we extend this function by $0$, it is quasi lower-semicontinuous
	on all of $\R^d$.
	Now, \cite[Lemma~1.5]{DalMaso1983} implies the existence of an increasing sequence $\{z_m\}_{m \in \N} \subset H^1(\R^d)$
	with $0 \le z_m$ and $z_m \nearrow y - \psi$ pointwise q.e.\ on $\R^d$.
	From $y - \psi = 0$ on $\R^d \setminus \Omega$,
	we have $z_m = 0$ q.e.\ on $\R^d \setminus \Omega$.
	Thus, $z_m \in H_0^1(\Omega)$, see \cite[Theorem~4.5]{HeinonenKilpelaeinenMartio1993}.
	This yields
	\begin{equation*}
		\int_\Omega (z_m - y + y_n) \, \d(1+\Delta w_n)
		\le
		\int_\Omega (y_n - \psi) \, \d(1+\Delta w_n)
		=
		0.
	\end{equation*}
	From $y_n \weakly y$ in $H_0^1(\Omega)$
	and $w_n \to w$ in $H_0^1(\Omega)$,
	we infer
	\begin{equation*}
		0
		\le
		\int_\Omega z_m \, \d(1+\Delta w)
		=
		\lim_{n \to \infty}
		\int_\Omega (z_m - y + y_n) \, \d(1+\Delta w_n)
		\le
		0.
	\end{equation*}
	Hence,
	\begin{equation*}
		\int_\Omega z_m \, \d(1+\Delta w)
		=
		0.
	\end{equation*}
	Finally, $\{z_m\}$ converges monotonically pointwise q.e.\ to $y - \psi$.
	The monotone convergence theorem implies
	\begin{equation*}
		\int_\Omega (y - \psi) \, \d( 1+ \Delta w)
		=
		\lim_{m \to \infty}
		\int_\Omega z_m \, \d(1+\Delta w)
		=
		0.
	\end{equation*}
	Therefore,
	$y - \psi = 0$ q.e.\ on $\fsupp(1 + \Delta w) =_q \Omega \setminus \hat\Omega$.
	Hence,
	$\Omega \setminus \hat\Omega \subset_q A(u)$
	and this yields the desired
	$I(u) \subset_q \hat\Omega$.
	
	``$\supset$'':
	This follows from 
	$\pBws S(u)\supset \pBss S(u)$
	and \cref{thm:characterization_pBss}.
\end{proof}

\cref{thm:characterization_pBss,thm:characterization_pBws}
show that $\pBws S(u) = \pBss S(u)$ for all $u \in H^{-1}(\Omega)$
without any regularity assumptions on the data.

\section{The strong-weak generalized derivative}
\label{sec:strong_weak_generalized_derivative}
In this section, we investigate $\pBsw S(u)$.
Since this generalized differential involves the WOT
for the convergence of the derivatives,
we expect that the resulting set is significantly larger than
$\pBss S(u)$.
In fact, we will see that capacitary measures enter the stage.
As a first result, we prove an upper bound.
\begin{lemma}
	\label{lem:upper_estimate_pBsw}
	Let $u \in H^{-1}(\Omega)$ be given.
	Then,
	\begin{equation}
		\label{est:upper_estimate_pBsw}
		\pBsw S(u)
		\subset
		\{
			L_\mu
			\mid
			\mu \in \MM_0(\Omega),
			\mu( I(u) ) = 0
			\text{ and }
			\mu = +\infty \text{ on } A_s(u)
		\}.
	\end{equation}
	Here, $\mu = +\infty$ on $A_s(u)$ is to be understood as
	\begin{equation}
		\label{eq:infty_on_A_s}
		\forall v \in H_0^1(\Omega) \cap L^2_\mu(\Omega):
		\qquad
		v = 0 \text{ q.e.\ on } A_s(u)
		.
	\end{equation}
\end{lemma}
\begin{proof}
	Let $L \in \pBsw S(u)$ be given.
	By definition, there is a sequence $\{u_n\} \subset D_S$
	with $u_n \to u$ in $H^{-1}(\Omega)$
	and
	$S'(u_n) \toWOT L$.
	From \cref{lem:completeness}
	we obtain $L = L_\mu$ for some $\mu \in \MM_0(\Omega)$.

	First, we show $\mu = +\infty$ on $A_s(u)$.
	Let $f \in H^{-1}(\Omega)$ be given.
	Then, $v_n := S'(u_n) f \weakly L_\mu f =: v$
	and $\abs{v_n} \weakly \abs{v}$ in $H_0^1(\Omega)$.
	For $\xi_n := -\Delta S(u_n) - u_n$ and $\xi := -\Delta S(u) - u$
	we have $\xi_n \to \xi$ in $H^{-1}(\Omega)$.
	It holds
	$\abs{v_n} = 0$ q.e.\ on $\fsupp(\xi_n) =_q A_s(u_n)$.
	This implies
	\begin{equation*}
		0
		=
		\lim_{n \to \infty} \dual{\xi_n}{\abs{v_n}}
		=
		\dual{\xi}{\abs{v}}.
	\end{equation*}
	Hence, $\abs{v} = 0$ q.e.\ on $\fsupp(\xi) =_q A_s(u)$.
	Since the range of $L_\mu$ is dense in
	$H_0^1(\Omega) \cap L^2_\mu(\Omega)$,
	see \cref{lem:density_image_L_mu},
	we have $\mu = +\infty$ on $A_s(u)$.

	It remains to show $\mu(I(u)) = 0$.
	Let $v \in H_0^1(I(u))$ with $0 \le v \le 1$
	and
	$\{v > 0\} =_q I(u)$
	be given, see \cref{thm:torsion_function_properties}.
	By \cref{lem:approximation_on_function_on_inactive_set},
	there exists a sequence $\{v_n\}$ with $v_n \to v$ in $H_0^1(\Omega)$
	and $v_n \in H_0^1(I(u_n))$.
	Therefore,
	$v_n = S'(u_n) (-\Delta v_n)$.
	Since $-\Delta v_n \to -\Delta v$ in $H^{-1}(\Omega)$,
	\cref{lem:SOT_und_WOT}~(ii)
	implies
	$v_n = S'(u_n) (-\Delta v_n) \weakly L_\mu(-\Delta v)$.
	Hence, $v = L_\mu(-\Delta v)$ and therefore, $v \in L^2_\mu(\Omega)$.
	Testing the associated weak formulation with $v$, we infer
	\begin{equation*}
		\int_\Omega \abs{\nabla{v}}^2 \, \d x
		+
		\int_\Omega v^2 \, \d\mu
		=
		\dual{-\Delta v}{v}
		=
		\int_\Omega \abs{\nabla{v}}^2 \, \d x
		.
	\end{equation*}
	Hence,
	$\int_\Omega v^2\,\d\mu = 0$
	and this means
	$v = 0$ $\mu$-a.e.\ on $\Omega$.
	Since
	$v > 0$ q.e.\ on $I(u)$
	and since $\mu$ does not charge polar sets,
	we have
	$v > 0$ $\mu$-a.e.\ on $I(u)$.
	This implies $\mu(I(u)) = 0$.
\end{proof}

To illustrate the meaning of $\mu = +\infty$ on $A_s(u)$,
we give some equivalent reformulations.
\begin{lemma}
	\label{lem:mu_infty_A_s}
	Let $u \in H^{-1}(\Omega)$ and $\mu \in \MM_0(\Omega)$ be given.
	Then, the following assertions are equivalent.
	\begin{enumerate}
		\item\label{it:lem_mu_infty_A_s_1}
			$\mu = +\infty$ on $A_s(u)$ in the sense of \eqref{eq:infty_on_A_s}.
		\item\label{it:lem_mu_infty_A_s_2}
			$\forall v \in H_0^1(\{w_\mu > 0\}) : \quad v = 0 \text{ q.e.\ on } A_s(u)$.
		\item\label{it:lem_mu_infty_A_s_3}
			$w_\mu = 0 \text{ q.e.\ on } A_s(u)$.
		\item\label{it:lem_mu_infty_A_s_4}
			$\mu \ge \infty_{A_s(u)}$.
	\end{enumerate}
\end{lemma}
\begin{proof}
	The equivalence between \ref{it:lem_mu_infty_A_s_1}
	and \ref{it:lem_mu_infty_A_s_2}
	follows from \cref{lem:completion_h01_l2}.
	From \cref{lem:v=0_on_w=0},
	we get that \ref{it:lem_mu_infty_A_s_2}
	and \ref{it:lem_mu_infty_A_s_3} are equivalent.

	Let us assume that \ref{it:lem_mu_infty_A_s_3} holds.
	By \cite[Proposition~3.4]{DalMasoGarroni1994},
	it holds $\mu(B)=+\infty$ for all Borel sets $B \subset \Omega$ with $\capa(B \cap \{w_\mu=0\})>0$
	and this gives \ref{it:lem_mu_infty_A_s_4}.

	Finally, \ref{it:lem_mu_infty_A_s_4}
	implies
	\ref{it:lem_mu_infty_A_s_3} by the comparison principle
	\cite[Theorem~2.10]{DalMasoMosco1986}.
\end{proof}

Note that if $u$ is a differentiability point of $S$,
then the right-hand side in \eqref{est:upper_estimate_pBsw}
reduces to $\{S'(u)\}$ and equality holds.

In the general case, the reverse inclusion in \eqref{est:upper_estimate_pBsw}
is much harder to obtain,
and we will prove it under some regularity assumption
on $\psi$.
However, in the very simple and artificial case that
the entire set $\Omega$ is biactive, i.e.,
$A(u) =_q \Omega$ and $A_s(u) =_q \emptyset$,
the equality in \eqref{est:upper_estimate_pBsw}
just follows from the density result in
\cref{lem:quasi_open_sets_dense}.
\begin{corollary}
	\label{cor:pBsw_biactive}
	Let $u \in H^{-1}(\Omega)$ be given such that
	$A(u) =_q \Omega$ and $A_s(u) =_q \emptyset$.
	Then,
	\begin{equation*}
		\pBsw S(u)
		=
		\{
			L_\mu
			\mid
			\mu \in \MM_0(\Omega)
		\}.
	\end{equation*}
	In particular, \eqref{est:upper_estimate_pBsw} holds with equality.
\end{corollary}
\begin{proof}
	The inclusion ``$\subset$'' is established in \cref{lem:upper_estimate_pBsw}
	and it remains to check
	``$\supset$''.
	From \cref{thm:characterization_pBss}, we have
	\begin{equation*}
		\pBss S(u)
		=
		\{
			L_{\hat\Omega}
			\mid
			\hat\Omega\subset\Omega \text{ is quasi-open}
		\}
		\subset
		\pBsw S(u).
	\end{equation*}
	Since the closure of the left-hand side
	w.r.t.\ WOT
	is
	$\{ L_\mu \mid \mu \in \MM_0(\Omega) \}$,
	see \cref{lem:quasi_open_sets_dense},
	and since
	$\pBsw S(u)$ is closed in WOT,
	see \cref{prop:properties_derivatives},
	this yields the claim.
\end{proof}

The verification of the reverse inclusion in \eqref{est:upper_estimate_pBsw} in the general
case is much more delicate.
The reason is that the density result \cref{lem:quasi_open_sets_dense}
is typically proved in a rather abstract way, i.e.,
it is not easy to obtain the approximating sequence of quasi-open sets $O_n$.
We are going to use the explicit construction from \cite{DalMasoMalusa1995}.
This, however, needs that $A(u_n)$ contains an open neighborhood of $A(u)$
and, therefore, we have to assume some regularity of $y$ and $\psi$.
We give some preparatory lemmas.
\begin{lemma}
	\label{lem:approximation_of_un}
	Let $u \in H^{-1}(\Omega)$ be given and define $y := S(u)$.
	We assume that $y \in C_0(\Omega)$,
	$\psi \in C(\bar\Omega) \cap H^1(\Omega)$.
	Further,
	we assume that
	$\psi \in H_0^1(\Omega)$
	or $\psi < 0$ on $\partial\Omega$.

	Then, there exists a sequence $\{u_n\} \subset H^{-1}(\Omega)$
	such that $u_n \to u$ in $H^{-1}(\Omega)$,
	$y_n := S(u_n)$ satisfies
	$y_n = \psi$ on $\{y < \psi + 1/n\}$
	and
	$\xi = -\Delta y-u=-\Delta y_n - u_n$.
	In particular,
	$\{y < \psi + 1/n\}$
	is an open neighborhood of
	$\{y = \psi\}$.
\end{lemma}
\begin{proof}
	Our strategy is to define $y_n$
	with the desired properties
	and to verify afterwards that $y_n$
	solves the obstacle problem with right-hand side $u_n := -\Delta y_n - \xi$.

	In the case that $\psi \in H_0^1(\Omega)$,
	we define
	$y_n := \max( y - 1/n , \psi )$.
	It is immediate that $y_n \in H_0^1(\Omega)$,
	$y_n \to y$ in $H_0^1(\Omega)$ and
	$y_n = \psi$ on $\{y < \psi + 1/n\}$.

	In the case that $\psi < 0$ on $\partial\Omega$,
	we have $\psi \le c$ on $\partial\Omega$ for some constant $c < 0$.
	From $y = 0$ on $\partial\Omega$,
	we find that the set $\{y = \psi\}$ has a positive distance
	to the boundary of $\Omega$.
	Thus, there exists a function $\varphi \in C_c^\infty(\Omega)$
	with $0 \le \varphi \le 1$ and $\varphi = 1$ on $\{y = \psi\}$.
	Now, we set
	$y_n := \max( y - \varphi/n, \psi )$.
	Again, we find
	$y_n \in H_0^1(\Omega)$,
	$y_n \to y$ in $H_0^1(\Omega)$ and
	$y_n = \psi$ on $\{y < \psi + 1/n\}$.

	Finally, we define $u_n := -\Delta y_n - \xi$.
	It is immediate that $u_n \to u$ in $H^{-1}(\Omega)$
	and we have to check that $y_n = S(u_n)$.
	The property $y_n \in K$ is immediate from the definition.
	From
	$\fsupp(\xi) =_q A_s(u) \subset_q \{y = \psi\} \subset \{y_n = \psi\}$,
	we infer
	$\fsupp(\xi) \subset_q \{z \ge y_n\}$ for all $z \in K$.
	Hence,
	\begin{equation*}
		\dual{-\Delta y_n - u_n}{z - y_n}
		=
		\dual{\xi}{z - y_n}
		=
		\int_\Omega (z - y_n) \, \d\xi
		\ge
		0.
	\end{equation*}
	This shows that $y_n = S(u_n)$.
\end{proof}
The next result shows that we can approximate solution operators
associated to Radon measures.
\begin{lemma}
	\label{lem:radon_measures_in_derivative}
	Let $u \in H^{-1}(\Omega)$ be given such that the assumptions of \cref{lem:approximation_of_un}
	are satisfied.
%	Let $\{u_n\}$
%	be the sequence given by \cref{lem:approximation_of_un}.
	Then, for every Radon measure $\mu \in \MM_0(\Omega)$
	with $\mu(I(u)) = 0$,
	the measure
	$\lambda = \mu + \infty_{A_s(u)}$
	satisfies
	\begin{equation*}
		L_{\lambda}
		\in
		\pBsw S(u).
	\end{equation*}
\end{lemma}
\begin{proof}
	Let $\mu$ be a given Radon measure as in the formulation of the lemma.
	We can use the construction of
	\cite[Theorem~2.5]{DalMasoMalusa1995}
	to obtain
	a sequence $\{E_m\}$ of compact subsets of $\Omega$
	with the property that each $E_m$ is contained
	in $\supp(\mu) + B_{1/m}$
	and $\infty_{E_m} \togamma \mu$.
	In particular, for all $n \in \mathbb{N}$,
	$E_m \subset \{y_n = \psi\}$ for $m$ large enough with $y_n = S(u_n)$,
	where the sequence $\{u_n\}$ is given by \cref{lem:approximation_of_un}.

	Now, we consider the sequence
	$\lambda_m := \infty_{E_m} + \infty_{A_s(u)}$.
	By \cref{thm:gamma_limit_of_sum_of_measures},
	we conclude that $\lambda_m \togamma \lambda$ as $m \to \infty$.
%	By the compactness properties of $\gamma$-convergence, 
%	see \cref{thm:compactness},
%	this sequence converges (up to a subsequence) towards some measure
%	$\hat\mu$ and it remains to show that
%	this measure is equivalent to $\lambda_n$.
%	\cite[Theorem~4.10]{DalMasoGarroni1994}
%	implies that
%	$\mu_m$ converges on the open set $\omega := \Omega \setminus (A_s(u) \setminus O_n)$
%	towards $\hat\mu$
%	and that
%	$\infty_{E_m}$ converges on $\omega$
%	towards $\mu$.
%	Since the sequences $\mu_m$ and $\infty_{E_m}$ coincide on $\omega$,
%	the limits have to be equivalent.
%	Thus, $\mu(O) = \hat\mu(O)$ for all quasi-open $O \subset \omega$
%	or, equivalently,
%	$\int_\omega v^2 \, \d\mu = \int_\omega v^2 \, \d\hat\mu$
%	for all $v \in H_0^1(\omega)$.
%	Next, we have to check that
%	$\int_\Omega v^2 \, \d\lambda_n = \int_\Omega v^2 \, \d\hat\mu$
%	for all $v \in H_0^1(\Omega)$.
%	If any of these integrals is finite,
%	we directly conclude $v = 0$ q.e.\ on $A_s(u) \setminus O_n$.
%	Hence, $v$ belongs to $H_0^1(\omega)$
%	and the equality of the integrals follows.
%	Therefore,
%	the measures $\lambda_n$ and $\hat\mu$
%	are equivalent and we conclude
%	$\mu_m \togamma \lambda_n$.
	Fix $n \in \mathbb{N}$.
	Then \cref{thm:characterization_pBss}
	implies that $L_{\lambda_m} \in \pBss S(u_n)$
	for all but finitely many $m \in \mathbb{N}$.
	Thus, the set inclusion $\pBss S(u) \subset \pBsw S(u)$ and 
	property (iii) from \cref{prop:properties_derivatives}
	imply that $L_\lambda \in \pBsw S(u_n)$ for all $n \in \mathbb{N}$.
	Applying \cref{prop:properties_derivatives} once more,
	we obtain that $L_\lambda \in \pBsw S(u)$
	and the claim follows.
\end{proof}

Now, we are able to give the main result of this section.
\begin{theorem}
	\label{thm:equality_for_pBsw}
	Let $u \in H^{-1}(\Omega)$ be given such that the assumptions of \cref{lem:approximation_of_un}
	are satisfied.
	Then, \eqref{est:upper_estimate_pBsw} holds with equality, i.e.,
	\begin{equation}
		\label{est:equality_pBsw}
		\pBsw S(u)
		=
		\{
			L_\mu
			\mid
			\mu \in \MM_0(\Omega),
			\mu( I(u) ) = 0
			\text{ and }
			\mu = +\infty \text{ on } A_s(u)
		\}.
	\end{equation}
\end{theorem}
\begin{proof}
%	sketch:
%	Approximate $\mu$ from below by Radon measures $\mu_m$.
%	Then, $\mu_m + \infty_{A_s(u) \setminus O_n}$
%	belongs to $\pBsw S(u_n)$.
%	Next, we show $\mu_m + \infty_{A_s(u) \setminus O_n} \togamma \mu$
%	(as in the proof of \cref{lem:radon_measures_in_derivative}).
%	Hence, $\mu \in \pBsw S(u_n)$.
%	By upper semicontinuity, we infer $\mu \in \pBsw S(u)$.\\
%	
	Let $\mu \in \MM_0(\Omega)$ 
	with $\mu(I(u))=0$ and $\mu=\infty$ on $A_s(u)$.
	By \cref{lem:approximation_Radon} 
	we find an increasing sequence $\{\mu_m\}$ of Radon measures with $\mu_m \togamma \mu$.
	Since $\mu_m\le \mu$, it holds $\mu_m(I(u))=0$.
	Thus, by \cref{lem:radon_measures_in_derivative},
	the measure $\lambda_{m}:=\mu_m+\infty_{A_s(u)}$ satisfies
	\begin{equation*}
		L_{\lambda_{m}}\in \pBsw S(u).
	\end{equation*}
	 Furthermore,
	 \cref{thm:gamma_limit_of_sum_of_measures} implies that
	 $\lambda_{m}\togamma \mu$ as $m \rightarrow \infty$.
	The closedness property of $\pBsw S$, 
	see \cref{prop:properties_derivatives}, 
	implies that $L_\mu \in \pBsw S(u)$.
\end{proof}

\section{The weak-weak generalized derivative}
\label{sec:weak_weak_generalized_derivative}
By means of an example,
we show that $\pBww S(u)$
can be surprisingly large.
In fact, we have seen that for
a Gâteaux point
$u \in D_S$
we have
$\pBss S(u) = \pBws S(u) = \pBsw S(u) = \{ S'(u) \}$,
see
\cref{thm:characterization_pBss,lem:upper_estimate_pBsw,thm:characterization_pBws}.
However, we will see that $\pBww S(u)$ might not be a singleton for $u \in D_S$.

We use the classical construction of \cite{CioranescuMurat1997}.
Therein, the authors construct a sequence
$\Omega_n$ of open subsets of $\Omega$
such that
the solution operators $L_{\Omega_n}$
of
\begin{equation*}
	-\Delta y_n = f \quad \text{in } \Omega_n
\end{equation*}
converge in WOT
to the solution operator
$L_c$
of
\begin{equation*}
	-\Delta y  + c \, y = f \quad \text{in } \Omega
\end{equation*}
for a positive constant $c > 0$.
We define $y = L_c 1$ and $y_n = L_{\Omega_n} 1$.
This yields $y_n \weakly y$.
We fix the obstacle
$\psi := 0$
and set
$u_n := -\Delta y_n - 2^{-n} \, \chi_{\Omega \setminus \Omega_n}$,
$u := -\Delta y$.
Then, it is clear that $y = S(u)$, $y_n = S(u_n)$ and $u_n \weakly u$.
Since $A(u) =_q \emptyset$, we have $u \in D_S$.
Similarly, we have
$A(u_n) =_q \{y_n = 0\} =_q \Omega \setminus \Omega_n$.
From
$\xi_n := -\Delta y_n - u_n = 2^{-n} \, \chi_{\Omega \setminus \Omega_n}$,
we have
$A_s(u_n) =_q \fsupp(\xi_n) =_q \Omega \setminus \Omega_n$,
since $\Omega \setminus \Omega_n$ is a finite union of balls (by construction).
Thus, $u_n \in D_S$
and $S'(u_n) = L_{\Omega_n}$.
By construction,
$L_{\Omega_n} \toWOT L_c$.
Hence,
$L_c \in \pBww S(u)$
although $u \in D_S$.

\section{Stationarity systems for the optimal control of the obstacle problem}
\label{sec:stationarity_systems}

In this section, we consider the optimal control of the obstacle problem
with control constraints
\begin{equation}
	\label{eq:optimal_control}
	\text{Minimize } J(y,u) \; \text{ with } y = S(u)
	\text{ and } u \in \Uad.
\end{equation}
Here,
$J : H_0^1(\Omega) \times L^2(\Omega) \to \R$ is given.
We assume that $J$ is Fréchet differentiable
with partial derivatives $J_y$ and $J_u$.
The admissible set $\Uad \subset L^2(\Omega)$ is assumed to be closed and convex.
We denote by $(y,u) \in H_0^1(\Omega) \times \Uad$ a local minimizer of \eqref{eq:optimal_control}.
A formal calculation leads to the stationarity systems
\begin{equation}
	\label{eq:opt_con_pBss}
	0 \in L\adjoint J_y(y,u) + J_u(y,u) + \NNUad(u)
	\qquad
	\text{for some } L \in \pBss S(u)
\end{equation}
and
\begin{equation}
	\label{eq:opt_con_pBws}
	0 \in L\adjoint J_y(y,u) + J_u(y,u) + \NNUad(u)
	\qquad
	\text{for some } L \in \pBsw S(u),
\end{equation}
where $\NNUad(u)$ denotes the normal cone (in the sense of convex analysis)
of $\Uad$ at $u$.
The goal of this section is the interpretation
of these systems
and a comparison with known optimality systems for \eqref{eq:optimal_control}.
For the discussion of \eqref{eq:opt_con_pBws}, we will assume that the characterization
\eqref{est:equality_pBsw} holds.
Recall that this is the case if $y$ and $\psi$ feature some additional regularity,
see \cref{thm:equality_for_pBsw}.

At this point, it is not clear whether any of these
stationarity conditions is necessary for local optimality.
If we would have defined the solution operator $S$
from $L^2(\Omega)$ to $H_0^1(\Omega)$,
then \eqref{eq:opt_con_pBws} would imply
that $0$ belongs to the sum of the Bouligand subdifferential
of the reduced objective $j(u) := J(S(u),u)$
at the point $u$
and the normal cone of $\Uad$ at $u$,
see the discussion in \cite[Section~4.2]{ChristofClasonMeyerWalther2017}.
However, the derivation of the generalized derivatives
for $S : L^2(\Omega) \to H_0^1(\Omega)$
is much more difficult
and postponed to future work.

We start by the interpretation of \eqref{eq:opt_con_pBss}.
\begin{lemma}
	\label{lem:opt_con_pBss}
	The condition \eqref{eq:opt_con_pBss}
	is equivalent to the existence of a quasi-closed set
	$A$ with $A_s(u) \subset_q A \subset_q A(u)$
	and of $p \in H_0^1(\Omega)$, $\nu \in H^{-1}(\Omega)$,
	$\lambda \in \NNUad(u)$
	such that
	\begin{align*}
		p + J_u(y,u) + \lambda&= 0,
		&
		p &\in H_0^1(\Omega \setminus A),
		\\
		-\Delta p + \nu &= J_y(y,u),
		&
		\nu & \in H^{-1}(\Omega)
		\text{ with }
		\dual{\nu}{v} = 0 \text{ for all }
		v \in H_0^1(\Omega \setminus A)
	\end{align*}
	hold.
\end{lemma}
\begin{proof}
	Let \eqref{eq:opt_con_pBss} be satisfied with some
	$L \in \pBss S(u)$.
	By \cref{thm:characterization_pBss},
	there exists a quasi-closed set $A$
	with $A_s(u) \subset_q A \subset_q A(u)$
	and
	$L = L_{\Omega \setminus A}$.
	Then, it is clear that
	$p := L\adjoint J_y(y,u) = L J_y(y,u)$,
	$\nu := J_y(y,u) + \Delta p$
	and
	$\lambda := -p - J_u(y,u)$
	satisfy the above system.

	The converse direction follows similarly.
\end{proof}
We note that the condition of \cref{lem:opt_con_pBss}
is a rather restrictive version of the system of M-stationarity
in \cite[Section~1.4]{Wachsmuth2014:2},
just use $\hat A_s :=_q A \setminus A_s(u)$, $\hat B :=_q \emptyset$
and $\hat A :=_q A(u) \setminus A$ therein.

The interpretation of \eqref{eq:opt_con_pBws} is much more challenging
and interesting.
\begin{lemma}
	\label{lem:opt_con_pBsw}
	The condition \eqref{eq:opt_con_pBws}
	implies
	the existence of
	$p \in H_0^1(\Omega)$, $\nu \in H^{-1}(\Omega)$, $\lambda \in \NNUad(u)$
	such that
	\begin{subequations}
		\label{eq:c_stationarity}
		\begin{align}
			\label{eq:c_stationarity_1}
			p + J_u(y,u) + \lambda &= 0
			&
			p &\in H_0^1(\Omega \setminus A_s(u))
			\\
			\label{eq:c_stationarity_2}
			-\Delta p + \nu &= J_y(y,u)
			&
			\nu & \in H^{-1}(\Omega)
			\text{ with }
			\dual{\nu}{v} = 0 \;\forall
			v \in H_0^1(\Omega \setminus A(u))
			\\
			\label{eq:c_stationarity_3}
			\dual{\nu}{p \, \varphi}
			&\ge 0
			\mathrlap{\qquad
			\forall \varphi \in W^{1,\infty}(\Omega)^+.}
		\end{align}
	\end{subequations}

	Conversely, if this system holds, if \eqref{est:equality_pBsw} holds
	and if there exists $\mu \in \MM_0(\Omega)$
	such that $\nu = p \, \mu$ in the sense
	$p \in L^2_\mu(\Omega)$ and 
	\begin{equation*}
		\dual{\nu}{w}
		=
		\int_\Omega p \, w \, \d\mu
		\qquad
		\forall w \in H_0^1(\Omega) \cap L^2_\mu(\Omega),
	\end{equation*}
	then \eqref{eq:opt_con_pBws} is satisfied.
\end{lemma}
\begin{proof}
	``\eqref{eq:opt_con_pBws}$\Rightarrow$\eqref{eq:c_stationarity}'':
	Let \eqref{eq:opt_con_pBws}
	be satisfied by some $L \in \pBsw S(u)$.
	From \eqref{est:upper_estimate_pBsw},
	there is $\mu \in \MM_0(\Omega)$
	with
	$L = L_\mu$,
	$\mu(I(u)) = 0$ and $\mu = +\infty$ on $A_s(u)$.
	We define
	$p := L\adjoint J_y(y,u) = L J_y(y,u)$,
	$\nu := J_y(y,u) + \Delta p$
	and $\lambda := -p - J_u(y,u)$.
	Then, $p = 0$ q.e.\ on $A_s(u)$, i.e., $p \in H_0^1(\Omega \setminus A_s(u))$.

	By definition of $\nu$ and $p$, we have
	\begin{equation*}
		\nu = J_y(y,u) + \Delta p = p \, \mu.
	\end{equation*}
	For $v \in H_0^1(\Omega \setminus A(u)) = H_0^1(I(u))$,
	we have $v \in L^2_\mu(\Omega)$ and we get
	\begin{equation*}
		\dual{\nu}{v}
		=
		\int_\Omega p \, v \, \d\mu
		=
		0
	\end{equation*}
	since $\mu = 0$ on $I(u)$ and $v$ lives only on $I(u)$.

	It remains to show $\dual{\nu}{p \, \varphi} = 0$
	for all $\varphi \in W^{1,\infty}(\Omega)^+$.
	We have $p \, \varphi \in H_0^1(\Omega)$
	and the pointwise boundedness of $\varphi$ gives $p \, \varphi \in L^2_\mu(\Omega)$.
	Thus,
	\begin{equation*}
		\dual{\nu}{p \, \varphi}
		=
		\int_\Omega p^2 \, \varphi \, \d\mu
		\ge
		0.
	\end{equation*}
	This shows that the above system is satisfied by $p$ and $\nu$.

	``\eqref{eq:c_stationarity}$\Rightarrow$\eqref{eq:opt_con_pBws}'':
	To prove the converse direction,
	let $p$, $\nu$, $\lambda$ and $\mu$ be given as in the assertion of the lemma.
	We will modify $\mu$ to construct
	another measure
	$\mu_2 \in \MM_0(\Omega)$, which satisfies the conditions on the right-hand side of \eqref{est:equality_pBsw},
	that is, $\mu_2(I(u)) = 0$ and $\mu_2 = +\infty$ on $A_s(u)$.
	First, we will set the measure to $+\infty$ in $A_s(u)$.
	Since $\{p = 0\} \setminus I(u) \supset_q A_s(u)$,
	we define
	$\mu_1 := \mu + \infty_{\{p = 0\} \setminus I(u)}$.
	We check that $\nu = p \, \mu_1$.
	Obviously, $p \in L^2_{\mu_1}(\Omega)$ since $p = 0$ q.e.\ on $\{p = 0\} \setminus I(u)$.
	Furthermore, for $w \in H_0^1(\Omega) \cap L^2_{\mu_1}(\Omega)$,
	we have $w \in L^2_\mu(\Omega)$, thus
	\begin{equation}
		\label{eq:equality_mu1}
		\dual{\nu}{w}
		=
		\int_\Omega p \, w \, \d\mu
		=
		\int_\Omega p \, w \, \d\mu_1
	\end{equation}
	since $\int_\Omega p \, w \, \d\infty_{\{p = 0\} \setminus I(u)} = 0$.

	Next, we define the Borel measure
	$\mu_2(B) := \mu_1(B \setminus I(u))$.
	Then, $\mu_2(I(u)) = 0$.
	It remains to show that we still have $\nu = p \, \mu_2$.
	The condition $p \in L^2_{\mu_2}(\Omega)$ is clear.

	We use \cref{lem:union_of_domains,lem:completion_h01_l2} to obtain
	\begin{equation}
		\label{eq:density_of_some_sobolev_spaces}
		\begin{aligned}
			\overline{
				H_0^1(\Omega) \cap L^2_{\mu_1}(\Omega)
				+
				H_0^1(I(u))
			}^{H_0^1(\Omega)}
			&=
			\overline{
				\overline{
					H_0^1(\Omega) \cap L^2_{\mu_1}(\Omega)
				}^{H_0^1(\Omega)}
				+
				H_0^1(I(u))
			}^{H_0^1(\Omega)}
			\\
			&=
			\overline{
				H_0^1(\{w_{\mu_1} > 0\})
				+
				H_0^1(I(u))
			}^{H_0^1(\Omega)}
			\\
			&=
			H_0^1\bigh(){ \{w_{\mu_1} > 0\} \cup I(u) }
			.
		\end{aligned}
	\end{equation}
	For all Borel sets $B \subset \Omega$
	with
	$\capa\bigh(){ B \cap \{p = 0\} \setminus I(u)} > 0$,
	we have
	\begin{equation*}
		\mu_2( B )
		=
		\mu_1( B \setminus I(u) )
		=
		\mu( B \setminus I(u) )
		+
		\infty_{\{p = 0\} \setminus I(u)}( B \setminus I(u) )
		=
		+\infty
		.
	\end{equation*}
	Thus,
	$w_{\mu_2} = 0$ q.e.\ on $\{p = 0\} \setminus I(u)$.
	By taking complements, this leads to
	\begin{equation*}
		\{w_{\mu_2} > 0\}
		\subset_q
		\Omega \setminus \bigh(){\{p = 0\} \setminus I(u)}
		=_q
		\Omega \setminus \bigh(){\{p = 0\} \cap A(u)}
		=_q
		\{p \ne 0\} \cup I(u)
		.
	\end{equation*}
	Moreover, from $p \in L^2_{\mu_1}(\Omega)$ and \cref{lem:v=0_on_w=0},
	we obtain
	\begin{equation}
		\label{eq:subset_w2}
		\{w_{\mu_2} > 0\}
		\subset_q
		\{w_{\mu_1} > 0 \} \cup I(u)
		.
	\end{equation}

	By combining \eqref{eq:density_of_some_sobolev_spaces},
	\eqref{eq:subset_w2}
	and \cref{lem:v=0_on_w=0},
	we find that every $v \in H_0^1(\Omega)^+ \cap L^2_{\mu_2}(\Omega)$
	belongs to $H_0^1\bigh(){ \{w_{\mu_1} > 0\} \cup I(u) }^+$
	and, therefore,
	there exist sequences
	$\{v^{(1)}_n\}_{n \in \N} \subset H_0^1(\Omega) \cap L^2_{\mu_1}(\Omega)$
	and
	$\{v^{(2)}_n\}_{n \in \N} \subset H_0^1(I(u))$
	with
	$v^{(1)}_n + v^{(2)}_n \to v$ in $H_0^1(\Omega)$.
	Further, from the second assertions of
	\cref{lem:union_of_domains,lem:completion_h01_l2},
	it can be seen that these sequences can be chosen such that additionally
	$0 \le v^{(1)}_n + v^{(2)}_n \le v$ q.e.\ on $\Omega$
	for all $n \in \N$.
	We can extract a subsequence (without relabeling),
	such that
	$v^{(1)}_n + v^{(2)}_n \to v$ pointwise q.e.,
	thus, pointwise $\mu_2$-a.e.
	Since $v \in L^2_{\mu_2}(\Omega)$,
	the dominated convergence theorem implies
	$v^{(1)}_n + v^{(2)}_n \to v$ in $L^2_{\mu_2}(\Omega)$.

	By construction,
	the functional $K : H_0^1(\Omega) \cap L^2_{\mu_2}(\Omega) \to \R$,
	given by
	\begin{equation*}
		K(w)
		:=
		\dual{\nu}{w} - \int_\Omega p \, w \, \d\mu_2
		,
	\end{equation*}
	vanishes on $H_0^1(I(u))$.

	Next, we show that $K$ vanishes also
	on $H_0^1(\Omega) \cap L^2_{\mu_1}(\Omega)$.
	We take $w \in H_0^1(I(u))$ with $0 \le w \le 1$ and $\{w > 0\} =_q I(u)$.
	Then, $\tilde w = \max( \min( w, p ), -w)$
	satisfies $\tilde w \in H_0^1(I(u)) \cap L^2_{\mu_1}(\Omega)$,
	since $\abs{\tilde w} \le \abs{p} \in L^2_{\mu_1}(\Omega)$.
	Hence, \eqref{eq:c_stationarity_2} and \eqref{eq:equality_mu1} imply
	\begin{equation*}
		0 = \dual{\nu}{\tilde w}
		=
		\int_\Omega p \, \tilde w \, \d\mu_1.
	\end{equation*}
	Now, $p \, \tilde w \ge 0$
	and $\{p \, \tilde w > 0\} =_q \{\abs{p} \ne 0\} \cap I(u)$.
	This shows $\mu_1( I(u) \cap \{p \ne 0\} ) = 0$.
	Thus, for arbitrary $w \in H_0^1(\Omega) \cap L^2_{\mu_1}(\Omega)$,
	we have
	\begin{equation*}
		\dual{\nu}{w}
		=
		\int_\Omega p \, w \, \d\mu_1
		=
		\int_{ \Omega \setminus I(u) } p \, w \, \d\mu_1
		=
		\int_{ \Omega \setminus I(u) } p \, w \, \d\mu_2
		=
		\int_\Omega p \, w \, \d\mu_2
		.
	\end{equation*}
	Hence, $K$ vanishes on $H_0^1(\Omega) \cap L^2_{\mu_1}(\Omega)$.

	Further,
	$K$ is linear and continuous w.r.t.\ the space $H_0^1(\Omega) \cap L^2_{\mu_2}(\Omega)$.
	Thus,
	\begin{equation*}
		K(v)
		=
		\lim_{n \to \infty}
		K\Bigh(){
			v^{(1)}_n + v^{(2)}_n
		}
		=
		0.
	\end{equation*}
	Hence, $K$ vanishes on
	$H_0^1(\Omega)^+ \cap L^2_{\mu_2}(\Omega)$
	and, by linearity,
	on the entire space
	$H_0^1(\Omega) \cap L^2_{\mu_2}(\Omega)$.
	This shows $\nu = p \, \mu_2$.

	Now,
	for
	$v \in H_0^1(\Omega) \cap L^2_{\mu_2}(\Omega)$,
	we have
	\begin{equation*}
		\int_\Omega p \, v \, \d\mu_2
		=
		\dual{\nu}{v}
		=
		\dual{J_y(y,u)}{v}
		-
		\int_\Omega \nabla p \, \nabla v \, \d x.
	\end{equation*}
	This shows
	$p = L_{\mu_2} J_y(y,u) = L_{\mu_2}\adjoint J_y(y,u)$.
	Hence, \eqref{eq:opt_con_pBws} is satisfied.
\end{proof}
Some remarks concerning \cref{lem:opt_con_pBsw}
are in order.
Under some regularity assumptions on the data
and on the objective
of the control problem \eqref{eq:optimal_control},
it was shown in
\cite{SchielaWachsmuth2013}
that the system
\eqref{eq:c_stationarity}
is satisfied at every local minimizer,
see also the comparison in \cite[Lemma~4.6]{Wachsmuth2014:2}.

Surprisingly,
the technique of \cite{SchielaWachsmuth2013}
even provides the additional condition
$\nu = p \, \mu$
after closer inspection.
Indeed, (by using the notation of \cite{SchielaWachsmuth2013}),
the adjoint state $p_c$ associated to a regularized problem solves
the semilinear equation
\begin{equation*}
	-\Delta p_c + c \, \max\nolimits_c'( \bar \lambda + (y_c - \psi)) \, p_c = J_y(y_c, u_c).
\end{equation*}
Here, $c > 0$ is a penalty parameter which will go to $\infty$.
Since the function $\max\nolimits_c$ is monotonically increasing,
we have
$c \, \max\nolimits_c'( \bar \lambda + (y_c - \psi)) \in \MM_0(\Omega)$.
\cref{thm:compactness} implies
that (along a subsequence)
$c \, \max\nolimits_c'( \bar \lambda + (y_c - \psi)) \togamma \mu$
for some $\mu \in \MM_0(\Omega)$
as $c \to \infty$.
Thus, the weak convergence $p_c \weakly p$ in $H_0^1(\Omega)$,
together with $y_c \to y$ in $H_0^1(\Omega)$
and $u_c \to u$ in $L^2(\Omega)$,
yields that the limit $p$ satisfies
\begin{equation*}
	-\Delta p + \mu \, p = J_y(y, u).
\end{equation*}
Hence, $\nu = p \, \mu$ in the sense of \cref{lem:opt_con_pBsw}.
This reasoning
and the results of \cite{SchielaWachsmuth2013}
imply
that \eqref{eq:opt_con_pBws}
is indeed satisfied by every local minimizer of \eqref{eq:optimal_control},
whenever \eqref{est:equality_pBsw} holds.

\section{Conclusion}
\label{sec:concl}
In this work we have shown that
the generalized derivatives
of the solution operator $S$ of the obstacle problem
are solution operators of relaxed Dirichlet problems.
In the case that the strong operator topology is considered,
the limit is a solution operator associated
to a quasi-open subset of $\Omega$,
whereas the usage of the weak operator topology
needs the notion of solution operators associated with capacitary measures.
By considering optimality systems
corresponding to the generalized derivatives of $S$,
we have seen that the notion of C-stationarity
from \cite{SchielaWachsmuth2013}
can be strengthened to a system including a capacitary measure.

%%fakesection: Acknowledgement
\subsection*{Acknowledgments}
The authors would like to thank Giuseppe Buttazzo for pointing out
the result of \cite[Lemma~4.12]{BucurButtazzoTrebeschi1999}
which led to the discovery of \cref{thm:gamma_limit_of_sum_of_measures}.
Further, Constantin Christof
pointed out the result of \cref{thm:torsion_function_properties}
($\fsupp(1 + \Delta w) =_q \Omega \setminus O$ for $w = L_O(1)$)
and this is gratefully acknowledged.
He gave a different, interesting proof based on differentiability properties
of the obstacle problem and this proof might appear elsewhere.

This work is supported by DFG grants
UL158/10-1
and
WA3636/4-1
within the
\href{https://spp1962.wias-berlin.de}{Priority Program SPP 1962}
(Non-smooth and Complementarity-based Distributed Parameter Systems: Simulation and Hierarchical Optimization).

%%fakesection: Bibliography
\ifbiber
	% We are cheating a little bit ;)
	% \renewcommand*{\bibfont}{\small}
	\printbibliography
\else
	\bibliographystyle{plainnat}
	\bibliography{references}
\fi
\end{document}